\documentclass[12pt]{article}

\usepackage{amsmath, amssymb, amsthm}
\usepackage{geometry}
\usepackage{hyperref}
\usepackage{booktabs}
\usepackage{array}
\usepackage{graphicx}
\usepackage{tikz}
\usepackage{pgfplots}
\pgfplotsset{compat=1.18}
\usepackage[ruled,linesnumbered,vlined,algo2e]{algorithm2e}
\usepackage{float}

\SetKwInput{KwIn}{Input}
\SetKwInput{KwOut}{Output}
\SetKwComment{Comment}{$\triangleright$\ }{}
\SetKw{KwRet}{return}
\usetikzlibrary{arrows.meta, positioning, shapes.geometric, fit,
                backgrounds, decorations.pathreplacing}
\newtheorem{theorem}{Theorem}[section]
\newtheorem{proposition}[theorem]{Proposition}
\newtheorem{corollary}[theorem]{Corollary}
\newtheorem{lemma}[theorem]{Lemma}
\newtheorem{definition}[theorem]{Definition}
\newtheorem{remark}[theorem]{Remark}
\newtheorem{example}[theorem]{Example}
\newtheorem{assumption}{Assumption}

\newcommand{\Rmax}{\mathbb{R}_{\max}}

\newcommand{\calR}{\mathcal{R}}
\newcommand{\calP}{\mathcal{P}}

\title{\textbf{Tropical Bi-Objective Pseudolinear Optimization as\\ 
Parametric Mean-Payoff Games}}

\author{
  Ibrahim Hassan\thanks{Mathematics Department, Kano University of Science 
  and Technology, Wudil, Nigeria. Email: [ihassan@kustwudil.edu.ng]} 
  \and
  Abdulhadi Aminu\thanks{Mathematics Department, Kano University of Science 
  and Technology, Wudil, Nigeria. Email: [abdulamin247@gmail.com]}
  \and 
  Lawal Muhammad\thanks{Mathematics Department, Northwest University, 
  Kano, Nigeria. Email: [l.muhammad@nwu.edu.ng]}
}

\date{\today}

\begin{document}

\maketitle

\begin{abstract}
We extend the parametric mean-payoff game framework of Parsons \emph{et~al.}~\cite{PSW2023} to bi-objective tropical pseudolinear
optimization with general two-sided constraints. The problem is to
simultaneously minimize two tropical pseudolinear objectives over the
feasible set of a general two-sided system $U \otimes x \oplus b \leq V
\otimes x \oplus d$, we characterize the Pareto front via a parametric
mean-payoff game in two parameters $(\lambda_1, \lambda_2)$.
The feasibility region $\mathcal{R}$ is convex and the Pareto front
$\mathcal{P}$ is a convex piecewise-linear curve with finitely many
breakpoints, these properties are natural extensions of the single-parameter
of~\cite{PSW2023} to two parameters.
In addition, we give as a new result, the joint denominator bound 
(Lemma~\ref{lem:k1k2_bound}): the cycle coefficients satisfy
$k_1(\gamma) + k_2(\gamma) \leq 2$ for any elementary cycle $\gamma$,
yielding $|\Delta| \leq 2$ for every $2\times 2$ Newton system, except in fully decoupled case and
implying that all breakpoints have half-integer coordinates for integer
data. Optimality and infeasibility certificates are given in terms of
the cycle structure of the parametric game. Two algorithms are
developed, a directional bisection algorithm
($O(n^2(n+m)\log M)$ per direction) and a Newton scheme tracing the
complete Pareto front via $2\times 2$ linear solves in at most
$|\mathcal{S}|$ steps, independent of $M$. The  directional bisection algorithm is pseudo-polynomial in $n,~m$ and $M$. The Newton scheme is independent of $M$ but requires up to $|S|$ steps, where $|S|$ is exponential in $n$. Lastly, we give numerical experiments on random instances to confirm
the directional bisection  complexity bound exactly, the Newton scheme's worst-case bound is not attained by random instances but is shown to be tight via explicit adversarial constructions.
\end{abstract}

\medskip
\noindent\textbf{Keywords:} Max-plus semiring; multi-objective combinatorial optimization; 
mean-payoff games; Pareto front; parametric optimization; piecewise-linear 
convexity.

\medskip
\noindent\textbf{AMS Classifications:} 15A80; 90C26; 90C29; 91A46.

\section{Introduction}

Tropical linear algebra has been studied since the foundational
work of Cuninghame-Green~\cite{CG1979} and Baccelli\emph{ et~ al}.~\cite{BCOQ1992},
and has grown into a rich area connecting linear algebra, algebraic
geometry, and combinatorial optimization. Single-objective tropical
optimization, including max-linear programming with two-sided
constraints~\cite{BA2009, B2010}, pseudolinear and pseudoquadratic
problems~\cite{K2012,K2014a,K2014b,K2015a,K2015b,K2017a,K2017b,AB2012},
and the connection to mean-payoff games~\cite{AGG2012,GKS2012,PSW2023} has been studied extensively. A detailed account of this literature
and how the present paper relates to it is given in
Section~\ref{sec:related}.

All of the above works consider a single scalar objective function. In many
practical applications, including the scheduling problem which is described later in this section,
multiple competing objectives arise simultaneously, and a single scalar
optimum is insufficient to capture the full trade-off between them. The
Pareto front, the set of objective-value pairs $(\lambda_1, \lambda_2)$ at
which one objective cannot be improved without worsening the other, is the
natural object of study. 

Consider a project consisting of $n$ tasks with start times
$x_1, \ldots, x_n \in \mathbb{R}$. The tasks are subject to
precedence constraints: task $j$ cannot begin until at least
$d_{ij}$ time units after task $i$ begins, giving
$x_j \geq x_i + d_{ij}$ for each relevant pair $(i,j)$.
In the tropical semiring $(\mathbb{R} \cup \{-\infty\}, \oplus,
\otimes)$, where $a \oplus b = \max(a,b)$ and $a \otimes b = a+b$,
this becomes $0 \otimes x_i \oplus (-d_{ij}) \leq 0 \otimes x_j$,
a two-sided tropical linear inequality. Collecting all such
constraints into matrix form gives the two-sided system
$U \otimes x \oplus b \leq V \otimes x \oplus d$,
which is precisely constraint~\eqref{eq:constraint} of the
present paper.

Two objectives arise naturally and conflict with each other.
The first is the maximum flow time
\[
  f_1(x) = \max_{j=1,\ldots,n}(x_j - r_j),
\]
where $r_j$ is the release date of task $j$. This measures
how long tasks are kept waiting relative to when they become
available, and the scheduler wishes to minimize it.
The second is the maximum lateness
\[
  f_2(x) = \max_{j=1,\ldots,n}(x_j + c_j - D),
\]
where $c_j$ is the processing time of task $j$ and $D$ is a
common project deadline. This measures how far the last task
overruns the deadline, and again the scheduler wishes to
minimize it. Starting tasks earlier reduces lateness but
increases flow time, and vice versa. These two objectives
cannot generally be minimized simultaneously.

Both $f_1$ and $f_2$ are special cases of a tropical pseudolinear
function of the general form
\[
  f(x) = x^- \otimes p \oplus q^- \otimes x
        = \max\Bigl(\max_j(p_j - x_j),\ \max_j(x_j - q_j)\Bigr).
\]
The first term $\max_j(p_j - x_j)$ is large when tasks start too
late relative to their target times, penalising delay. The second
term $\max_j(x_j - q_j)$ is large when tasks start too early
relative to their release dates, penalising idleness. Setting
$(p_1)_j$ and $(q_1)_j$ to the appropriate parameters for the
flow-time objective, and $(p_2)_j$ and $(q_2)_j$ for the lateness
objective, gives
\[
  f_1(x) = x^- \otimes p_1 \oplus q_1^- \otimes x,
  \qquad
  f_2(x) = x^- \otimes p_2 \oplus q_2^- \otimes x,
\]
with precise data specified in Definition~\ref{def:biobjective}.
Minimising $f_1$ and $f_2$ simultaneously subject to the
two-sided constraint is precisely Problem~\ref{def:biobjective}
of this paper. The Pareto front of the problem gives the
complete trade-off curve between flow time and lateness, which entails that
for each achievable level of flow time $\lambda_1$, it
identifies the minimum lateness $\lambda_2$ that can be
simultaneously achieved, and vice versa.
The scheduler can then select any point on this curve
according to the priorities of the project, with full
knowledge of what is and is not simultaneously achievable.

This paper initiates the study of tropical
bi-objective pseudolinear optimization with general two-sided constraints.
Our main contributions are as follows.

\begin{enumerate}
  \item 
  We reformulate the bi-objective problem as a parametric mean-payoff
  game depending on \emph{two} parameters $(\lambda_1, \lambda_2)$
  simultaneously, by introducing two disjoint groups of $\lambda$-arcs
  (one for each objective) into the game matrices.
  The two-dimensional parametric structure, verifying that both
  Assumptions~\ref{ass:max}--\ref{ass:min} of the single-objective are preserved and that
  the duality theorem also extends to the two-parameter setting.

  \item 
  We prove that the feasibility region
  $\calR = \{(\lambda_1,\lambda_2) : \exists\, x \text{ feasible with }
  f_1(x) \leq \lambda_1,\, f_2(x) \leq \lambda_2\}$ is convex, and that
  the Pareto front $\calP$ is a convex piecewise-linear curve with finitely
  many breakpoints (Theorem~\ref{thm:pareto_structure}(i)--(ii)).
  These properties follow from the halfspace representation of $\calR$
  and the finiteness of the strategy set $\mathcal{S}$, both of which
  carry over from the one-parameter analysis of~\cite{PSW2023} to two
  parameters.

  \item \
  We prove a denominator bound: for integer data, every breakpoint of
  $\calP$ has coordinates that are integer multiples of $\tfrac{1}{2}$
  (Theorem~\ref{thm:pareto_structure}(iii)).
  The key technical ingredient is Lemma~\ref{lem:k1k2_bound}, which
  establishes the \emph{joint} bound $k_1(\gamma) + k_2(\gamma) \leq 2$
  for any elementary cycle $\gamma$ in the two-parameter game.
  This bound and the resulting $|\Delta| \leq 2$ for the $2\times 2$
  Newton system requires a combinatorial argument about the shared
  hub-node structure of the two $\lambda$-groups that has no
  single-parameter analogue in~\cite{PSW2023}.

  \item 
  We develop a directional bisection algorithm that reduces the
  bi-objective problem to a finite sequence of scalar parametric
  problems via weighted scalarization, each solved by the
  method of~\cite{PSW2023}, with total cost
  $O(n^2(n+m)\log M)$ per direction
  (Proposition~\ref{prop:termination}).

  \item 
  We establish optimality and infeasibility certificates in terms of
  the cycle structure of the parametric game
  (Propositions~\ref{prop:optimality} and~\ref{prop:infeasibility}),
  extending the single-objective certificates of~\cite{PSW2023} to
  the two-parameter setting.

  \item 
  We introduce a Newton scheme that traces the complete Pareto front
  by jumping directly between consecutive breakpoints.
  At each step, the next breakpoint is located by the
  \emph{first-crossing criterion} (Definition~\ref{def:next_strategy}),
  which identifies the geometrically correct next halfspace boundary
  as $c$ decreases; this requires solving a $2 \times 2$ linear system at each step.
  The non-trivial two-dimensional geometry: the $2\times 2$ solve, the
  first-crossing criterion, and the proof that strategies are traversed
  without repetition (Proposition~\ref{prop:newton_termination}) is actually new and has no direct analogue in the one-parameter
  case of~\cite{PSW2023}, where the corresponding system is $1\times 1$.
\end{enumerate}

Contributions~1, 2, 4, and~5 are natural and competent extensions of
~\cite{PSW2023} to two parameters.
Contribution~3 (the joint denominator bound via Lemma~\ref{lem:k1k2_bound})
and Contribution~6 (the $2\times 2$ Newton scheme with first-crossing
criterion) constitute the new technical result.

The paper is organized as follows. Section~\ref{sec:related} surveys
related work on tropical optimization, mean-payoff games, and
multi-objective optimization. Section~\ref{sec:prelim} recalls the
necessary background on tropical linear algebra and mean-payoff games,
including  the two assumptions ensuring well-posedness, and the duality
theorem. Section~\ref{sec:problem} states the bi-objective problem
formally, introduces the parametric two-sided system, and analyses the
specific structure of the resulting mean-payoff game: it verifies
Assumptions~\ref{ass:max} and~\ref{ass:min} for our matrices,
establishes the cycle mean formula (Lemma~\ref{lem:cyclemean}), and
derives the bi-objective analogues of the game-theoretic results of~\cite{GKS2012}
(Proposition~\ref{prop:biobjective_game}). Section~\ref{sec:structure}
establishes the structural properties of the feasibility region and the
Pareto front. Section~\ref{sec:certificates} provides optimality and
infeasibility certificates. Section~\ref{sec:algorithm} develops the
directional bisection algorithm. Section~\ref{sec:example}
provides two detailed worked examples. Section~\ref{sec:newton} develops the
Newton scheme for tracing the full Pareto front. Section~\ref{sec:computation}
presents numerical experiments on random instances.

\section{Related Work}
\label{sec:related}

We survey related work on tropical optimization and mean-payoff games
along three threads: single-objective tropical optimization, the
game-theoretic connection, and multi-objective optimization.

\subsection*{Single-Objective Tropical Optimization}

Tropical optimization, minimizing a tropical function subject to tropical
constraints has been studied systematically since the foundational work
of Zimmermann~\cite{Z1976, Z1981} and Cuninghame-Green~\cite{CG1979},
whose notions of minimax algebra and max-plus semiring provide the
algebraic foundation for this paper.
Butkovi\v{c}~\cite{B2010} gave a comprehensive treatment of
max-linear systems, including existence, uniqueness, and structure of
solutions. Butkovi\v{c} and Aminu~\cite{BA2009} introduced
max-linear programming with two-sided constraints
$A \otimes x \oplus c = B \otimes x \oplus d$, motivated by
multiprocessor scheduling, and established solution methods for integer
data; that work is the direct precursor of the constraint structure used
in this paper.

Krivulin~\cite{K2012, K2014a, K2014b, K2015a, K2015b, K2017a, K2017b}
solved a range of tropical pseudolinear and pseudoquadratic optimization
problems with special constraint structures (including location analysis
and project scheduling) using algebraic methods based on the Kleene star
and spectral theory. Those results cover scalar objective functions with
specific constraint types, and their solution maps (alcoved polyhedra)
inform our discussion of the optimal solution map in
Corollary~\ref{cor:solution_map}. Aminu and
Butkovi\v{c}~\cite{AB2012} extended max-linear programming to
non-linear objective functions via a heuristic approach.

The most direct predecessor of the present work is Parsons \emph{et~al.}~\cite{PSW2023}, who showed that tropical pseudolinear and
pseudoquadratic optimization with general two-sided constraints is
equivalent to a parametric mean-payoff game in a single parameter
$\lambda$. They developed a bisection algorithm (terminating in
$O(\log(M))$ iterations for data bounded by $M$) and a Newton scheme
(terminating in at most $|\mathcal{S}|$ steps) for the single-objective
problem, and proved that the optimal value is always a half-integer for
integer data. The present paper extends the game-theoretic reformulation,
the convexity and piecewise-linearity structure, the bisection algorithm,
and the optimality certificates of~\cite{PSW2023} to the bi-objective
setting with two parameters $(\lambda_1, \lambda_2)$ simultaneously.
The joint denominator bound (Lemma~\ref{lem:k1k2_bound}) and the
two-dimensional Newton scheme (Definition~\ref{def:next_strategy},
Proposition~\ref{prop:newton_termination}) are the principal new
technical contributions beyond this extension.

\subsection*{Mean-Payoff Games and the Tropical Connection}

Mean-payoff games were introduced by Ehrenfeucht and
Mycielski~\cite{EM1979}, who proved that positional strategies suffice
and that the game value exists and equals the cycle mean. The connection
between tropical linear algebra and mean-payoff games was first
established by Akian \emph{et~al.}~\cite{AGG2012}, who showed
that tropical polyhedra are equivalent to winning regions in mean-payoff
games. Gaubert and Gunawardena~\cite{GG1998} proved the duality theorem
for min-max functions, which underlies the saddle-point characterization
of the feasibility region used in Proposition~\ref{prop:saddle}. The
parametric MPG framework connecting tropical optimization to game theory
was developed by Gaubert \emph{et~al.}~\cite{GKS2012} for tropical
linear-fractional programming; their Theorems 1, 3 and 4 are the direct
analogues of Proposition~\ref{prop:biobjective_game} in the present paper,
extended from one to two parameters.

\subsection*{Multi-Objective Optimization}

Multi-objective optimization and Pareto front computation are classical
topics in operations research. In the tropical geometry literature,
Develin and Sturmfels~\cite{DS2004} and Joswig and Loho~\cite{JL2016}
have studied tropical convex hulls and tropical hyperplane arrangements,
which are related to multi-dimensional tropical geometry. However, those
works address the combinatorial and geometric structure of tropical convex
sets, not the problem of \emph{optimizing} tropical pseudolinear
objectives subject to general two-sided constraints. To the best of our knowledge, no prior work has
addressed multi-objective tropical optimization in the sense studied here. The structural results of the
present paper: convexity of the feasibility region
(Theorem~\ref{thm:pareto_structure}(i)), piecewise linearity of the
Pareto front (Theorem~\ref{thm:pareto_structure}(ii)), and the
denominator bound (Theorem~\ref{thm:pareto_structure}(iii)) are new
in the tropical semiring. The parametric mean-payoff game framework we
develop is a natural bi-objective analogue of the single-parameter game
of~\cite{PSW2023}, replacing a one-dimensional bisection over $\lambda$
with a two-dimensional characterization of the set
$\mathcal{R} \subseteq \mathbb{R}^2$.

\section{Preliminaries}
\label{sec:prelim}

\subsection{Tropical Linear Algebra}

The \emph{tropical semiring} is $\Rmax = \mathbb{R} \cup \{-\infty\}$
equipped with operations $a \oplus b = \max(a,b)$ and $a \otimes b = a + b$
for all $a, b \in \Rmax$. The additive identity is $-\infty$ and the
multiplicative identity is $0$. These operations extend to matrices and
vectors in the natural way, for matrices $A = (a_{ij})$ and $B = (b_{kl})$
of compatible sizes,
\[
  (A \oplus B)_{ij} = a_{ij} \oplus b_{ij}, \qquad
  (A \otimes B)_{ij} = \bigoplus_k a_{ik} \otimes b_{kj}
                     = \max_k (a_{ik} + b_{kj}).
\]
The \emph{unit matrix} $I$ is the square matrix with diagonal entries $0$
and off-diagonal entries $-\infty$.

For $a \in \mathbb{R} \cup \{-\infty\} \cup \{+\infty\}$, its
\emph{conjugate} is defined as $a^- = -a$ if $a \in \mathbb{R}$,
$a^- = +\infty$ if $a = -\infty$, and $a^- = -\infty$ if $a = +\infty$.
For a column vector $x = (x_i)$, its conjugate is the row vector
$x^- = (x_i^-)$. The conjugate of a matrix $A$ with entries in
$\mathbb{R} \cup \{-\infty\}$ is the matrix $\overline{A}$ with entries
$(\overline{A})_{ij} = a_{ji}^-$.

With a matrix $A \in (\mathbb{R} \cup \{-\infty\})^{n \times n}$ we
associate a weighted digraph $D_A = (N, E, w)$ with node set
$N = \{1, \ldots, n\}$, arc set $E = \{(i,j) : a_{ij} \neq -\infty\}$,
and weight function $w(i,j) = a_{ij}$. The \emph{maximum cycle mean} of
$A$ is
\[
  \rho(A) = \max_\gamma \frac{w(A,\gamma)}{l(\gamma)},
\]
where the maximum is over all elementary cycles $\gamma$ in $D_A$,
$w(A,\gamma)$ denotes the total weight of $\gamma$, and $l(\gamma)$
its length.

\begin{proposition}[{\cite[Theorem~1.6.25]{B2010}}]
\label{prop:residuation}
Let $A \in \Rmax^{m \times n}$, $x \in \Rmax^n$, and
$b \in (\mathbb{R} \cup \{+\infty\})^m$. Then
\[
  A \otimes x \leq b \quad \text{if and only if} \quad
  x \leq \overline{A} \otimes b.
\]
\end{proposition}

\subsection{Two-Sided Systems and Min-Max Functions}

A \emph{two-sided tropical system} is an inequality of the form
$A \otimes x \leq B \otimes x$, where $A, B \in \Rmax^{m \times n}$.
By Proposition~\ref{prop:residuation}, this is equivalent to
$x \leq \overline{A} \otimes (B \otimes x)$, which can be written
componentwise as
\[
  x_j \leq \min_{k:\, a_{kj} \in \mathbb{R}}
  \left( -a_{kj} + \max_{l:\, b_{kl} \in \mathbb{R}} (b_{kl} + x_l) \right),
  \quad j = 1, \ldots, n.
\]
The right-hand side defines a \emph{min-max function}
\begin{equation}
\label{eq:minmax}
  f_j(x) = \min_{k \in [m]} \left( -a_{kj} + \max_{l \in [n]}
  (b_{kl} + x_l) \right), \quad j = 1, \ldots, n,
\end{equation}
so that the two-sided system $A \otimes x \leq B \otimes x$ is equivalent
to $x \leq f(x)$. A min-max function is isotone ($x \leq y \Rightarrow
f(x) \leq f(y)$) and additively homogeneous ($f(\lambda + x) = \lambda +
f(x)$ for all $\lambda \in \mathbb{R}$), hence nonexpansive in the
sup-norm ($\|f(x)-f(y)\|_\infty \leq \|x-y\|_\infty$ for all
$x,y \in \mathbb{R}^n$). It is also piecewise affine (affine on each
region of a finite polyhedral partition of $\mathbb{R}^n$).

We are interested in the long-run average behaviour of the iterates
$f^k(x)$ as $k \to \infty$. The following theorem of Kohlberg guarantees
the existence of an \emph{invariant half-line} of $f$.

\begin{theorem}[{\cite[Theorem~2]{GKS2012}}, Kohlberg~\cite{K1980}]
\label{thm:kohlberg}
Let $f : \mathbb{R}^n \to \mathbb{R}^n$ be a nonexpansive piecewise-affine
function. Then there exist $v \in \mathbb{R}^n$ and $\chi \in \mathbb{R}^n$
such that
\[
  f(v + t\chi) = v + (t+1)\chi, \quad \forall\, t \geq T,
\]
for some large enough $T \in \mathbb{R}$. The function $t \mapsto v + t\chi$
is an invariant half-line of $f$.
\end{theorem}

Using the nonexpansiveness of $f$, it follows that the limit~\cite{GKS2012}
\begin{equation}
\label{eq:cycletime}
  \chi(f) = \lim_{k \to \infty} \frac{f^k(x)}{k}
\end{equation}
exists, is independent of the starting point $x \in \mathbb{R}^n$, and
equals the growth rate $\chi$ of any invariant half-line. The vector
$\chi(f)$ is called the \emph{cycle-time vector} of $f$. We write
$\chi(\overline{A}B)$ for the cycle-time vector of the min-max function
$x \mapsto \overline{A} \otimes (B \otimes x)$.

\subsection{Mean-Payoff Games}

A \emph{mean-payoff game} (MPG) is played on a weighted bipartite
directed graph $G$ whose node set is the disjoint union $[m] \cup [n]$,
where nodes in $[m]$ belong to player Max and nodes in $[n]$ belong to
player Min. The game is defined by two matrices $A, B \in \Rmax^{m \times n}$.
When the pawn is at node $k \in [m]$ of Max, he chooses an arc to some
node $l \in [n]$ of Min and receives payment $b_{kl}$. When the pawn is
at node $j \in [n]$ of Min, she chooses an arc to some node $i \in [m]$
of Max and pays $-a_{ij}$ to Max. Arcs with weight $-\infty$ are
prohibited. Player Max seeks to maximize, and player Min to minimize,
the long-run average payment per turn.

We make the following assumptions, which ensure that both players always
have at least one move available at every node.

\begin{assumption}
\label{ass:max}
For all $k \in [m]$ there exists $l \in [n]$ such that $b_{kl} \neq -\infty$.
\end{assumption}

\begin{assumption}
\label{ass:min}
For all $j \in [n]$ there exists $i \in [m]$ such that $a_{ij} \neq -\infty$.
\end{assumption}

A \emph{positional strategy} for Max is a mapping $\sigma: [m] \to [n]$
such that $b_{i\sigma(i)} \neq -\infty$ for all $i \in [m]$. A
positional strategy for Min is a mapping $\tau: [n] \to [m]$ such that
$a_{\tau(j)j} \neq -\infty$ for all $j \in [n]$. We denote by
$\mathcal{S}$ and $\mathcal{T}$ the sets of positional strategies of Max
and Min respectively.

When both players reveal their positional strategies $\sigma$ and $\tau$,
the play proceeds in the subgraph $G^{\sigma,\tau}$ where each node has
exactly one outgoing arc. This is a sunflower digraph, each node has a
unique path leading to a unique cycle. The mean weight per turn of that
cycle is the payoff. The weight of a Min-to-Max arc from $j$ to $i$ is
$-a_{ij}$, and the weight of a Max-to-Min arc from $i$ to $j$ is $b_{ij}$.
The total weight of a cycle is therefore the sum of $-a_{ij}$ over its
Min-to-Max arcs plus the sum of $b_{ij}$ over its Max-to-Min arcs.

The invariant half-line of Theorem~\ref{thm:kohlberg} determines a pair
of optimal positional strategies $(\sigma^*, \tau^*)$ for the two players:
these are the strategies under which $f$ acts as $f^{\sigma^*}$ and
$f^{\tau^*}$ simultaneously along the half-line $t \mapsto v + t\chi$
for $t$ large enough, as described in~\cite{GKS2012}.

For a fixed positional strategy $\sigma$ of Max, the min-max function
$f$ reduces to the min-plus linear function
\[
  f^\sigma_j(x) = \min_{i \in [m]} (-a_{ij} + b_{i\sigma(i)} + x_{\sigma(i)}),
\]
and for a fixed strategy $\tau$ of Min it reduces to the max-plus linear
function
\[
  f^\tau_j(x) = -a_{\tau(j)j} + \max_{l \in [n]} (b_{\tau(j)l} + x_l).
\]
The cycle-time vectors of these one-player functions have explicit
combinatorial descriptions in terms of cycle means~\cite{GKS2012}.
For a max-plus linear function $x \mapsto Ex$ with associated digraph
$D(E)$, define the \emph{maximal cycle mean} of a strongly connected
component $S$ as
\[
  \mu^{\max}_S(E) = \max_{\text{cycles }\gamma \subseteq S}
  \frac{\sum_{(i,j)\in\gamma} e_{ij}}{|\gamma|},
\]
where $|\gamma|$ denotes the number of arcs in $\gamma$. Then
\begin{equation}
\label{eq:cycletime_maxplus}
  \chi^{\max}_i(E) = \max\{ \mu^{\max}_S(E) : i \text{ accesses } S\},
\end{equation}
meaning that $\chi^{\max}_i(E)$ is the largest cycle mean accessible
from node $i$~\cite{GKS2012}. Analogously, for min-plus linear functions,
the cycle-time vector is the \emph{minimal} cycle mean accessible from
each node~\cite{GKS2012}. These formulas are the key tool for computing
cycle-time vectors and for establishing our certificates in
Section~\ref{sec:certificates}.

\begin{theorem}[{\cite[Theorem~1]{GKS2012}}]
\label{thm:mpg_value}
Let $A, B \in \Rmax^{m \times n}$ satisfy Assumptions~\ref{ass:max}
and~\ref{ass:min}. Then there exists a vector $\chi \in \mathbb{R}^n$
and a pair of positional strategies $\sigma^* \in \mathcal{S}$ and
$\tau^* \in \mathcal{T}$ such that $\sigma^*$ secures a mean profit of
at least $\chi_j$ from any starting node $j$ of Min regardless of Min's
strategy, and $\tau^*$ secures a mean loss of no more than $\chi_j$
regardless of Max's strategy. The vector $\chi$ is uniquely determined by
\[
  \chi_j = \min_\tau \max_\sigma\, \Phi_{A,B}(j,\tau,\sigma)
          = \max_\sigma \min_\tau\, \Phi_{A,B}(j,\tau,\sigma),
\]
where the finite-horizon payoff $\Phi_{A,B}(j,\tau,\sigma)$ is the mean
weight per turn of the unique cycle in $G^{\sigma,\tau}$ accessible from
$j$. Moreover, $\chi_j = \chi_j(\overline{A}B)$, the $j$-th component
of the cycle-time vector of the min-max function~\eqref{eq:minmax}.
\end{theorem}

The following duality result, derived as a corollary of
Theorem~\ref{thm:kohlberg}, relates the cycle-time vector of the
min-max function to those of the one-player subgames.

\begin{theorem}[{\cite[Theorem~4]{GKS2012}}]
\label{thm:duality}
Let $A, B \in \Rmax^{m \times n}$ satisfy Assumptions~\ref{ass:max}
and~\ref{ass:min}, and let $\mathcal{S}$ and $\mathcal{T}$ be as above.
Then
\begin{equation}
\label{eq:duality}
  \max_{\sigma \in \mathcal{S}} \chi(f^\sigma)
  = \chi(f)
  = \min_{\tau \in \mathcal{T}} \chi(f^\tau),
\end{equation}
where $f$ is the min-max function~\eqref{eq:minmax}, $f^\sigma$ is the
min-plus linear function with Max strategy $\sigma$ fixed, and $f^\tau$
is the max-plus linear function with Min strategy $\tau$ fixed.
\end{theorem}

\begin{proposition}[{\cite[Theorem~5]{GKS2012}}]
\label{prop:solvability}
Let $A, B \in \Rmax^{m \times n}$ satisfy Assumptions~\ref{ass:max}
and~\ref{ass:min}. Then $\chi_j(\overline{A}B) \geq 0$ if and only if
there exists $x \in \Rmax^n$ such that $A \otimes x \leq B \otimes x$
and $x_j \neq -\infty$.
\end{proposition}

The proof of Proposition~\ref{prop:solvability} uses the invariant
half-line of Theorem~\ref{thm:kohlberg}: the vector $x$ is constructed
by taking $x_i = v_i + t\chi_i$ for $t$ large enough when $\chi_i \geq 0$,
and $x_i = -\infty$ otherwise, where $t \mapsto v + t\chi$ is the
invariant half-line of $f$~\cite{GKS2012}.

\section{Problem Formulation and MPG Representation}
\label{sec:problem}

\subsection{The Bi-Objective Problem}

We study the following tropical bi-objective pseudolinear optimization problem.

\begin{definition}
\label{def:biobjective}
(Tropical Bi-Objective Pseudolinear Optimization).
Given $p_1, p_2 \in \Rmax^n$, $q_1, q_2 \in (\mathbb{R} \cup \{+\infty\})^n$, 
$U, V \in \Rmax^{m \times n}$, $b, d \in \Rmax^m$, find the Pareto front of
\begin{equation}
\label{eq:biobjective}
  \min \begin{pmatrix} f_1(x) \\ f_2(x) \end{pmatrix} 
  = \min \begin{pmatrix} 
      x^- \otimes p_1 \oplus q_1^- \otimes x \\ 
      x^- \otimes p_2 \oplus q_2^- \otimes x 
    \end{pmatrix}
\end{equation}
subject to
\begin{equation}
\label{eq:constraint}
  U \otimes x \oplus b \leq V \otimes x \oplus d.
\end{equation}
\end{definition}

\begin{assumption}
\label{ass:wellposed}
(Well-posedness). The feasibility region
$\calR \subset \mathbb{R}^2$ is non-empty: there exists at least one
finite $x \in \mathbb{R}^n$ satisfying~\eqref{eq:constraint} that lies
on the Pareto front of~\eqref{eq:biobjective}.
\end{assumption}

\begin{remark}
Assumption~\ref{ass:wellposed} is a standing hypothesis throughout the
paper. It ensures that the Pareto front $\calP$ is non-empty and that
the algorithms have meaningful output. The infeasibility
certificate of Proposition~\ref{prop:infeasibility} provides an
explicit test for whether this assumption holds.
\end{remark}

\begin{remark}
The asymmetry in the spaces for $p_k$ and $q_k$ is deliberate and
follows from the conjugate operation. In the term $x^-\otimes p_k
= \max_j(p_{kj} - x_j)$, setting $p_{kj} = -\infty$ renders
component $j$ inactive (contributing $-\infty$ to the max), so
$p_k \in \Rmax^n = (\mathbb{R}\cup\{-\infty\})^n$ is the natural
space. In the term $q_k^-\otimes x = \max_j(x_j - q_{kj})$,
setting $q_{kj} = +\infty$ renders component $j$ inactive
(contributing $-\infty$ to the max), while $q_{kj} = -\infty$
would give $+\infty$ and make $f_k$ identically $+\infty$, so
$q_k \in (\mathbb{R}\cup\{+\infty\})^n$ is the correct space.
When $q_{1i} = +\infty$ (resp.\ $q_{2i} = +\infty$), the term $x_i$ does
not appear in $f_1$ (resp.\ $f_2$). When $p_{1i} = -\infty$ (resp.\
$p_{2i} = -\infty$), the term $x_i^-$ does not contribute to $f_1$
(resp.\ $f_2$). The constraint~\eqref{eq:constraint} is a general
two-sided tropical affine inequality describing tropical polyhedra,
as studied in~\cite{AGG2012, GKS2012, PSW2023}.
\end{remark}

\subsection{Parameterisation and Feasibility Region}
\label{sec:parameterisation}

We introduce a pair of parameters $(\lambda_1, \lambda_2) \in \mathbb{R}^2$ 
as upper bounds on the two objective functions. Problem~\eqref{eq:biobjective}--\eqref{eq:constraint} 
is equivalent to characterizing the set

\begin{equation}
\label{eq:feasibility_region}
  \calR = \left\{ (\lambda_1, \lambda_2) \in \mathbb{R}^2 : 
  \exists\, x \in \mathbb{R}^n \text{ finite, feasible, with } 
  f_1(x) \leq \lambda_1 \text{ and } f_2(x) \leq \lambda_2 \right\},
\end{equation}

and its lower boundary, the \emph{Pareto front}:

\begin{equation}
\label{eq:pareto_front}
  \calP = \left\{ (\lambda_1, \lambda_2) \in \calR : 
  \nexists\, (\lambda_1', \lambda_2') \in \calR \text{ with } 
  \lambda_1' < \lambda_1 \text{ and } \lambda_2' < \lambda_2 \right\}.
\end{equation}

We now show how the feasibility condition for a given
$(\lambda_1,\lambda_2)$ reduces to a single two-sided tropical system.
We introduce a homogenising scalar variable $t \in \mathbb{R}$ and set
$z = (x^\top, t)^\top \in \mathbb{R}^{n+1}$.

\medskip
\textbf{Step 1: Original constraint.}
The feasibility constraint~\eqref{eq:constraint} is
$U \otimes x \oplus b \leq V \otimes x \oplus d$.
Writing $b = b \otimes t^0$ and $d = d \otimes t^0$ (i.e.\ treating the
constant vectors as coefficients of the scalar $t$ evaluated at $t=0$),
this becomes the row group
\[
  \begin{pmatrix} U & b \end{pmatrix} \otimes z
  \;\leq\;
  \begin{pmatrix} V & d \end{pmatrix} \otimes z,
\]
which contributes the first $m$ rows of $A$ and $B(\lambda_1,\lambda_2)$.

\medskip
\textbf{Step 2: Bounding $f_1(x) \leq \lambda_1$.}
Since $f_1(x) = x^-\otimes p_1 \oplus q_1^-\otimes x$, the condition
$f_1(x) \leq \lambda_1$ splits into two parts:
\begin{align*}
  x^-\otimes p_1 \leq \lambda_1
  &\;\Longleftrightarrow\; \max_j(p_{1j} - x_j) \leq \lambda_1
   \;\Longleftrightarrow\; p_{1j} \leq \lambda_1 + x_j \text{ for all } j \\
  &\;\Longleftrightarrow\; (-\infty \mid p_1)\otimes z \leq (\lambda_1 I \mid -\infty)\otimes z,
\end{align*}
contributing the $p_1$ row group (rows $m+1$ to $m+n$), and
\begin{align*}
  q_1^-\otimes x \leq \lambda_1
  &\;\Longleftrightarrow\; \max_j(x_j - q_{1j}) \leq \lambda_1
   \;\Longleftrightarrow\; x_j - q_{1j} \leq \lambda_1 \text{ for all } j \\
  &\;\Longleftrightarrow\; (q_1^- \mid -\infty)\otimes z \leq (-\infty \mid \lambda_1)\otimes z,
\end{align*}
contributing the $q_1$ row group (a single scalar row $m+n+1$), where
$t$ plays the role of the homogenising variable carrying $\lambda_1$.

\medskip
\textbf{Step 3: Bounding $f_2(x) \leq \lambda_2$.}
By the same argument applied to $f_2$ and $\lambda_2$, we obtain the
$p_2$ row group (rows $m+n+2$ to $m+2n+1$) and the $q_2$ row group
(row $m+2n+2$), replacing $\lambda_1$ with $\lambda_2$ throughout.

\medskip
Collecting all five row groups into a single two-sided system
$A \otimes z \leq B(\lambda_1, \lambda_2) \otimes z$, a finite solution
$x \in \mathbb{R}^n$ together with any $t \in \mathbb{R}$ satisfies the
combined system if and only if $x$ is feasible for~\eqref{eq:constraint}
and $f_k(x) \leq \lambda_k$ for $k=1,2$. Therefore, $(\lambda_1,\lambda_2)
\in \calR$ if and only if the following parametric two-sided system

\begin{equation}
\label{eq:two_sided}
  A \otimes z \leq B(\lambda_1, \lambda_2) \otimes z
\end{equation}

has a finite solution $z \in \mathbb{R}^{n+1}$, where

\begin{equation}
\label{eq:AB_matrices}
  A = \begin{pmatrix} U & b \\ -\infty & p_1 \\ q_1^- & -\infty \\ 
  -\infty & p_2 \\ q_2^- & -\infty \end{pmatrix}, \qquad
  B(\lambda_1, \lambda_2) = \begin{pmatrix} V & d \\ \lambda_1 I & -\infty \\ 
  -\infty & \lambda_1 \\ \lambda_2 I & -\infty \\ -\infty & \lambda_2 
  \end{pmatrix}.
\end{equation}

Here $A$ and $B(\lambda_1, \lambda_2)$ are $(m+2n+2) \times (n+1)$
matrices, with rows partitioned into five groups: $m$ original constraint
rows, $n$ rows for $p_1$, one row for $q_1$, $n$ rows for $p_2$, and one
row for $q_2$. The parameters $\lambda_1$ and $\lambda_2$ enter $B$
through separate, disjoint groups of rows.

\subsection{Structure of the Bi-Objective Game}
\label{sec:game_structure}

We now analyse the specific mean-payoff game defined by $A$ and
$B(\lambda_1, \lambda_2)$ in~\eqref{eq:AB_matrices}. This game has
$N = m + 2n + 2$ nodes for Max (the rows) and $n+1$ nodes for Min
(the columns, corresponding to $x_1, \ldots, x_n$ and the
homogenising variable $t$).

\begin{proposition}
\label{prop:assumptions}
The matrices $A$ and $B(\lambda_1, \lambda_2)$ defined
in~\eqref{eq:AB_matrices} satisfy Assumptions~\ref{ass:max}
and~\ref{ass:min} provided that each row of $U, V$ has at least one
finite entry, each of $p_1, p_2$ has at least one finite entry, and
each of $q_1, q_2$ has at least one finite entry.
\end{proposition}

\begin{proof}
Assumption~\ref{ass:max} (every row of $B(\lambda_1,\lambda_2)$ has a
finite entry) holds because: constraint rows of $B$ inherit finite
entries from $V$ and $d$, the $p_k$ rows have $\lambda_k$ on the
diagonal, which is finite for any $(\lambda_1,\lambda_2) \in \mathbb{R}^2$,
the $q_k$ rows have $\lambda_k$ in the column corresponding to $t$,
again finite. Assumption~\ref{ass:min} (every row of $A$ has a finite
entry) holds because, constraint rows inherit from $U$ and $b$, the
$p_k$ rows have finite entries from $p_k$, the $q_k$ rows have finite
entries from $q_k^-$.
\end{proof}

The key structural feature of this game is how the two parameters
$\lambda_1$ and $\lambda_2$ enter the arc weights. Inspecting
$B(\lambda_1,\lambda_2)$, arcs in the $C_2 \cup C_3$ group carry weight
$\lambda_1$, arcs in the $C_4 \cup C_5$ group carry weight $\lambda_2$,
and all constraint arcs ($C_1$) carry weights from the data matrices
$V, d$, independent of $\lambda_1$ and $\lambda_2$. This is the
bi-objective analogue of the single parameter $\lambda$ entering through
one group of arcs in the single-objective construction of~\cite{GKS2012}.

\begin{lemma}[Cycle mean formula]
\label{lem:cyclemean}
Let $\gamma$ be a cycle in the bipartite digraph of the game defined by
$A$ and $B(\lambda_1, \lambda_2)$. Following the arc-weight
convention of~\cite{GKS2012},
arc weights are $-a_{ij}$ for Min-to-Max arcs and $b_{ij}$ for
Max-to-Min arcs. Let $k_1(\gamma)$ denote the number
of arcs in $\gamma$ carrying weight $\lambda_1$ (from group $C_2 \cup
C_3$), $k_2(\gamma)$ the number carrying $\lambda_2$ (from group $C_4
\cup C_5$), and $s(\gamma)$ the sum of all remaining arc weights in
$\gamma$, that is, $-a_{ij}$ for data Min-to-Max arcs and $b_{ij}$
for data Max-to-Min arcs. Then the mean weight per turn of $\gamma$ is
\begin{equation}
\label{eq:cyclemean}
  \mu(\gamma, \lambda_1, \lambda_2) 
  = \frac{k_1(\gamma)\,\lambda_1 + k_2(\gamma)\,\lambda_2 + s(\gamma)}
         {|\gamma|},
\end{equation}
where $|\gamma|$ is the total number of arcs in $\gamma$.
\end{lemma}

\begin{proof}
The total weight of $\gamma$ is $k_1(\gamma)\lambda_1 +
k_2(\gamma)\lambda_2 + s(\gamma)$ by definition of the arc weights,
where the $\lambda_{1,2}$-carrying Max-to-Min or Min-to-Min arcs contribute $k_{1,2}(\gamma)\lambda_{1,2}$
and the remaining arcs contribute $s(\gamma)$.
Dividing by the cycle length $|\gamma|$ gives the mean weight per turn.
\end{proof}

\begin{figure}[ht]
\centering
\begin{tikzpicture}[
  minnode/.style={circle,draw=black,fill=white,
    minimum size=0.72cm,inner sep=1pt,thick,font=\small},
  lam1node/.style={rectangle,draw=blue!70!black,fill=blue!8,
    minimum size=0.72cm,inner sep=1pt,thick,font=\small},
  lam2node/.style={rectangle,draw=red!70!black,fill=red!8,
    minimum size=0.72cm,inner sep=1pt,thick,font=\small},
  consnode/.style={rectangle,draw=gray!60!black,fill=gray!8,
    minimum size=0.72cm,inner sep=1pt,thick,font=\small},
  ->,>=Stealth,thick
]
\node[minnode](v1) at (0, 3.2){$v_1$};
\node[minnode](v2) at (0, 1.2){$v_2$};
\node[minnode](v0) at (0,-1.2){$v_0$};
\node[consnode](R1) at (-3.5, 3.2){$R_1$};
\node[consnode](R2) at (-3.5, 1.2){$R_2$};
\node[lam1node](L1a) at (3.8, 4.0){$L_1^{(1)}$};
\node[lam1node](L1b) at (3.8, 2.2){$L_1^{(2)}$};
\node[lam1node](L1c) at (3.8, 0.2){$L_1^{(0)}$};
\node[lam2node](L2a) at (3.8,-0.8){$L_2^{(1)}$};
\node[lam2node](L2b) at (3.8,-2.0){$L_2^{(2)}$};
\node[lam2node](L2c) at (3.8,-3.2){$L_2^{(0)}$};

\draw (v1) -- node[above,font=\scriptsize]{$0$} (R1);
\draw (v2) -- node[above,font=\scriptsize]{$0$} (R2);
\draw (v0) to[bend left=15] node[above,font=\scriptsize]{$-2$} (L1a);
\draw (v2) -- node[above,font=\scriptsize]{$0$} (L1c);
\draw (v0) to[bend right=10] node[below,font=\scriptsize]{$0$} (L2b);
\draw (v1) to[bend right=15] node[below,font=\scriptsize]{$0$} (L2c);

\draw (R1) -- node[above,font=\scriptsize]{$1$} (v2);
\draw (R2) -- node[above,font=\scriptsize]{$1$} (v1);
\draw (R2) to[bend left=20] node[right,font=\scriptsize]{$-2$} (v0);
\draw (L1a) -- node[above,font=\scriptsize]{$\lambda_1$} (v1);
\draw (L1b) -- node[above,font=\scriptsize]{$\lambda_1$} (v2);
\draw (L1c) -- node[above,font=\scriptsize]{$\lambda_1$} (v0);
\draw (L2a) to[bend left=10] node[above,font=\scriptsize]{$\lambda_2$} (v1);
\draw (L2b) -- node[above,font=\scriptsize]{$\lambda_2$} (v2);
\draw (L2c) -- node[above,font=\scriptsize]{$\lambda_2$} (v0);

\draw[gray!50,dashed,rounded corners]
  (-4.8, 0.5)--(-4.8,4.0)--(-2.2,4.0)--(-2.2,0.5)--cycle;
\node[font=\small,gray!60!black] at (-3.5,-0.2) {$C_1$};
\draw[blue!40,dashed,rounded corners]
  (2.8,-0.2)--(2.8,4.8)--(5.0,4.8)--(5.0,-0.2)--cycle;
\node[font=\small,blue!70!black] at (3.9,5.2) {$C_2\cup C_3\;(\lambda_1)$};
\draw[red!40,dashed,rounded corners]
  (2.8,-3.8)--(2.8,-0.4)--(5.0,-0.4)--(5.0,-3.8)--cycle;
\node[font=\small,red!70!black] at (3.9,-4.2) {$C_4\cup C_5\;(\lambda_2)$};
\end{tikzpicture}
\caption{The parametric MPG for Example~1 ($n=m=2$). Min nodes
$v_1,v_2$ (circles) correspond to variables $x_1,x_2$; $v_0$ is the
homogenising variable $t$. Max nodes are partitioned into the constraint
group $C_1 = \{R_1,R_2\}$ (grey), the $\lambda_1$ group
$C_2\cup C_3 = \{L_1^{(1)},L_1^{(2)},L_1^{(0)}\}$ (blue), and the
$\lambda_2$ group $C_4\cup C_5 = \{L_2^{(1)},L_2^{(2)},L_2^{(0)}\}$
(red). Min-to-Max arc weights are $-a_{ij}$ (negatives of $A$ entries);
Max-to-Min arc weights are $b_{ij}$ (entries of $B(\lambda_1,\lambda_2)$).}
\label{fig:mpg_example1}
\end{figure}
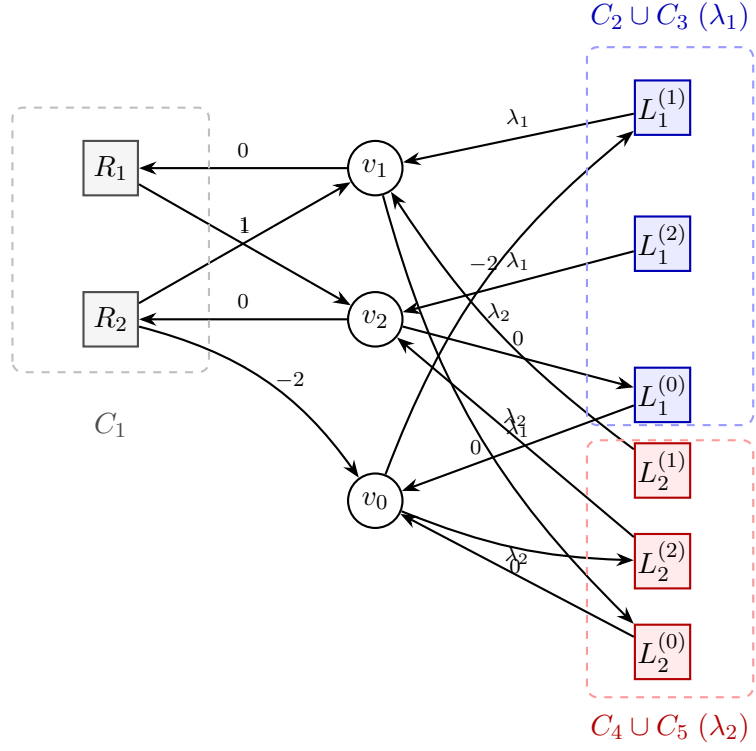

\begin{example}[Cycle mean illustration]
\label{ex:cyclemean}
Consider the cycle
$\gamma: v_0 \to L_1^{(1)} \to v_1 \to L_2^{(0)} \to v_0$
in Figure~\ref{fig:mpg_example1}. Following the arc weight convention
of Lemma~\ref{lem:cyclemean}, its four arcs have weights:
\[
  v_0 \xrightarrow{-p_{11}} L_1^{(1)} \xrightarrow{\lambda_1} v_1
  \xrightarrow{-a_{L_2^{(0)},v_1}} L_2^{(0)} \xrightarrow{\lambda_2} v_0.
\]
With $p_{11}=2$ and $a_{L_2^{(0)},v_1}=0$ from the matrices of
Example~1 (defined explicitly in Section~8.2), the data arc weights are $-p_{11}=-2$ and $0$ respectively.
Thus $k_1(\gamma)=1$, $k_2(\gamma)=1$,
$s(\gamma)=-2+0=-2$, and $|\gamma|=4$. By
Lemma~\ref{lem:cyclemean},
\[
  \mu(\gamma,\lambda_1,\lambda_2)
  = \frac{1\cdot\lambda_1 + 1\cdot\lambda_2 + (-2)}{4}
  = \frac{\lambda_1+\lambda_2-2}{4}.
\]
The non-negativity condition $\mu(\gamma,\lambda_1,\lambda_2)\geq 0$
gives the halfspace $\lambda_1+\lambda_2 \geq 2$, which is precisely
the binding constraint defining the Pareto front $\calP$.
\end{example}

 Since
$k_1(\gamma), k_2(\gamma) \geq 0$, the cycle mean is a non-decreasing
linear function of $(\lambda_1, \lambda_2)$. The condition
$\mu(\gamma, \lambda_1, \lambda_2) \geq 0$ defines the halfspace
\[
  k_1(\gamma)\,\lambda_1 + k_2(\gamma)\,\lambda_2 \geq -s(\gamma),
\]
and $\mathcal{R}$ is the intersection of all such halfspaces over all
cycles $\gamma$ and all Min strategies $\tau$. This is the geometric
origin of the convexity of $\mathcal{R}$ proved in
Theorem~\ref{thm:pareto_structure}(i).

The two groups $C_2 \cup C_3$ (carrying $\lambda_1$) and $C_4 \cup C_5$
(carrying $\lambda_2$) are \emph{disjoint} subsets of the Max nodes.
A single elementary cycle in the bipartite digraph can pass through at
most one node of Max per Min node it visits. Therefore, each elementary
cycle satisfies
\[
  k_1(\gamma) + k_2(\gamma) \leq |\gamma|/2,
\]
with equality only if the cycle alternates perfectly between the
$\lambda$-carrying groups and Min nodes. This is the reason
for the denominator bound in Theorem~\ref{thm:pareto_structure}(iii),
the denominator $|\gamma|$ of the breakpoint coordinates is bounded by
the number of nodes, giving the factor of $\tfrac{1}{2}$.

Applying Theorems~\ref{thm:mpg_value} and~\ref{thm:duality} to the
game with matrices $A$ and $B(\lambda_1, \lambda_2)$, we obtain the
following bi-objective analogues of Theorems 1, 3 and 4 of~\cite{GKS2012}.

\begin{proposition}
\label{prop:biobjective_game}
Let $A$ and $B(\lambda_1, \lambda_2)$ be as in~\eqref{eq:AB_matrices},
with the data satisfying the conditions of
Proposition~\ref{prop:assumptions}. Then:
\begin{enumerate}
\item[(i)] For each fixed $(\lambda_1, \lambda_2) \in \mathbb{R}^2$,
the mean-payoff game has a uniquely determined value vector
$\chi(\overline{A}B(\lambda_1, \lambda_2)) \in \mathbb{R}^{n+1}$,
whose $j$-th component $\chi_j = \chi_j(\overline{A}B(\lambda_1,\lambda_2))$
gives the long-run average payoff per turn when the game starts from
Min node $v_j$. Moreover, there exist optimal positional strategies
$\sigma^* \in \mathcal{S}$ and $\tau^* \in \mathcal{T}$ such that:
$\sigma^*$ secures Max a mean profit of at least $\chi_j$ from any
starting node $v_j$, regardless of Min's strategy; and $\tau^*$
secures a mean loss of no more than $\chi_j$ from any starting node
$v_j$, regardless of Max's strategy.

\item[(ii)] The value satisfies the duality:
\[
  \max_{\sigma \in \mathcal{S}} \chi(f^\sigma(\lambda_1,\lambda_2))
  = \chi(\overline{A}B(\lambda_1,\lambda_2))
  = \min_{\tau \in \mathcal{T}} \chi(f^\tau(\lambda_1,\lambda_2)).
\]

\item[(iii)] For each strategy $\sigma$ of Max, fixing $\sigma$
reduces the two-player game to a min-plus one-player game $G^\sigma$
in which Min controls all moves. Since Min seeks to minimize Max's
payoff, the cycle-time $\chi_j(f^\sigma(\lambda_1,\lambda_2))$ is the
\emph{minimal} cycle mean over all cycles in $G^\sigma$ accessible
from $j$, not maximal, because it is Min who optimizes against the
fixed $\sigma$. A cycle $\gamma$ in $G^\sigma$ is \emph{accessible
from $j$} if there exists a directed path from $j$ to some node of
$\gamma$ in $G^\sigma$. Since each node in $G^\sigma$ has exactly one
outgoing arc (determined by $\sigma$), every node eventually leads to
exactly one cycle in $G^\sigma$, and the minimum is taken over all
such cycles reachable from $j$. Explicitly:
\[
  \chi_j(f^\sigma(\lambda_1,\lambda_2))
  = \min_{\substack{\text{cycles } \gamma \text{ in } G^\sigma \\ \text{accessible from } j}}
    \frac{k_1(\gamma)\lambda_1 + k_2(\gamma)\lambda_2 + s(\gamma)}{|\gamma|},
\]
where $k_1(\gamma), k_2(\gamma) \geq 0$ are the $\lambda_1$-count and
$\lambda_2$-count of $\gamma$ in $G^\sigma$, $s(\gamma)$ is the sum
of its data weights, and $|\gamma|$ is its length. Since $k_1(\gamma), k_2(\gamma) \geq 0$ for each elementary cycle in $G^\sigma$ 
(as established in Lemma~4.5), each cycle mean $(k_1(\gamma)\lambda_1 + 
k_2(\gamma)\lambda_2 + s(\gamma))/|\gamma|$ is linear in $(\lambda_1, \lambda_2)$. 
Consequently, $(\lambda_1, \lambda_2) \mapsto \chi_j(f^\sigma(\lambda_1, \lambda_2))$ 
is a minimum of finitely many linear functions of $(\lambda_1, \lambda_2)$, hence 
piecewise linear, continuous, and componentwise non-decreasing.

\item[(iv)] The feasibility condition $(\lambda_1, \lambda_2) \in \calR$
holds if and only if all components of the cycle-time vector
$\chi(\overline{A}B(\lambda_1, \lambda_2))$ are non-negative, i.e.\
$\chi_j(\overline{A}B(\lambda_1,\lambda_2)) \geq 0$ for all
$j \in [n+1]$.
\end{enumerate}
\end{proposition}

\begin{proof}
\textbf{ (i).} For each fixed $(\lambda_1,\lambda_2)\in\mathbb{R}^2$,
the matrices $A$ and $B(\lambda_1,\lambda_2)$ satisfy
Assumptions~\ref{ass:max} and~\ref{ass:min} by
Proposition~\ref{prop:assumptions}. Theorem~\ref{thm:mpg_value} then
guarantees that the mean-payoff game with these matrices has a unique
value vector $\chi = \chi(\overline{A}B(\lambda_1,\lambda_2))\in\mathbb{R}^{n+1}$
and optimal positional strategies $\sigma^*\in\mathcal{S}$ and
$\tau^*\in\mathcal{T}$. The value $\chi_j$ is the $j$-th component of
the cycle-time vector of the min-max function $x\mapsto\overline{A}\otimes
(B(\lambda_1,\lambda_2)\otimes x)$, giving the long-run average payoff
per turn starting from Min node $v_j$.

\textbf{(ii).} Fix $(\lambda_1,\lambda_2)\in\mathbb{R}^2$. By
Theorem~\ref{thm:duality} applied to $A$ and $B(\lambda_1,\lambda_2)$:
\[
  \max_{\sigma\in\mathcal{S}}\chi(f^\sigma(\lambda_1,\lambda_2))
  = \chi(\overline{A}B(\lambda_1,\lambda_2))
  = \min_{\tau\in\mathcal{T}}\chi(f^\tau(\lambda_1,\lambda_2)).
\]
The left equality says that Max can guarantee the full game value by
playing the optimal strategy $\sigma^*$ that maximizes his one-player
cycle-time. The right equality says that Min can guarantee the same
value by playing the optimal strategy $\tau^*$ that minimizes Max's
one-player cycle-time. The two sides meeting at the same value is the
minimax property of the MPG.

\textbf{(iii).} Fix a strategy $\sigma\in\mathcal{S}$ of Max.
The min-max function $f$ reduces to the min-plus linear function
$f^\sigma_j(x) = \min_{i\in[m]}(-a_{ij}+b_{i\sigma(i)}+x_{\sigma(i)})$.
In the one-player game $G^\sigma$, each node has exactly one outgoing
arc, so each Min node $j$ has a unique directed path leading to a
unique cycle $\gamma$ in $G^\sigma$ accessible from $j$.
By~\eqref{eq:cycletime_maxplus}, the cycle-time $\chi_j(f^\sigma)$
equals the minimal cycle mean over all cycles $\gamma$ in $G^\sigma$
accessible from $j$:
\[
  \chi_j(f^\sigma(\lambda_1,\lambda_2))
  = \min_{\substack{\text{cycles }\gamma\text{ in }G^\sigma \\ \text{accessible from }j}}
    \frac{k_1(\gamma)\lambda_1 + k_2(\gamma)\lambda_2 + s(\gamma)}{|\gamma|},
\]
where the mean weight of each cycle $\gamma$ in $G^\sigma$ is given
explicitly by Lemma~\ref{lem:cyclemean} as
$(k_1(\gamma)\lambda_1+k_2(\gamma)\lambda_2+s(\gamma))/|\gamma|$,
with $k_1(\gamma),k_2(\gamma)\geq 0$ counting the $\lambda_1$-arcs and
$\lambda_2$-arcs in $\gamma$, and $s(\gamma)$ the sum of its data weights.

\textbf{(iv).} By construction of $A$ and $B(\lambda_1,\lambda_2)$
in~\eqref{eq:AB_matrices}, $(\lambda_1,\lambda_2)\in\calR$ if and only
if the two-sided system $A\otimes z\leq B(\lambda_1,\lambda_2)\otimes z$
has a finite solution $z\in\mathbb{R}^{n+1}$ (as shown in
Section~\ref{sec:parameterisation}). By
Proposition~\ref{prop:solvability} applied to these matrices, this
holds if and only if $\chi_j(\overline{A}B(\lambda_1,\lambda_2))\geq 0$
for all $j\in[n+1]$.
\end{proof}

We now define the key scalar function that governs feasibility and
encodes the full structure of $\calR$ and $\calP$.

Define the function $\Phi: \mathbb{R}^2 \to \mathbb{R}$ by

\begin{equation}
\label{eq:phi}
  \Phi(\lambda_1, \lambda_2) = \min_i \chi_i\!\left(\overline{A}\, 
  B(\lambda_1, \lambda_2)\right),
\end{equation}

where $\chi_i$ denotes the $i$-th component of the cycle-time vector 
of the min-max function defined by $A$ and $B(\lambda_1, \lambda_2)$. 
By Proposition~\ref{prop:biobjective_game}(iv), $(\lambda_1, \lambda_2) \in \calR$ 
if and only if $\Phi(\lambda_1, \lambda_2) \geq 0$.

Define also the strategy-restricted functions, for a strategy $\tau$ of 
Min and a strategy $\sigma$ of Max,
\[
  \Phi_\tau(\lambda_1, \lambda_2) = \min_i \chi_i\!\left(\overline{A_\tau}\, 
  B(\lambda_1, \lambda_2)\right), \qquad
  \Phi^\sigma(\lambda_1, \lambda_2) = \min_i \chi_i\!\left(\overline{A}\, 
  B_\sigma(\lambda_1, \lambda_2)\right),
\]
where $(A_\tau)_{ij} = a_{ij}$ if $i = \tau(j)$ and $-\infty$ otherwise, 
and $(B_\sigma)_{ij} = b_{ij}$ if $j = \sigma(i)$ and $-\infty$ otherwise.

\begin{proposition}
\label{prop:saddle}
Let $\mathcal{T}$ be the set of all strategies of player Min and 
$\mathcal{S}$ be the set of all strategies of player Max. Then
\begin{equation}
\label{eq:saddle}
  \min_{\tau \in \mathcal{T}} \Phi_\tau(\lambda_1, \lambda_2) 
  = \Phi(\lambda_1, \lambda_2) 
  = \max_{\sigma \in \mathcal{S}} \Phi^\sigma(\lambda_1, \lambda_2).
\end{equation}
\end{proposition}

\begin{proof}
By Proposition~\ref{prop:biobjective_game}(i), for each fixed
$(\lambda_1,\lambda_2)\in\mathbb{R}^2$, the game value vector
$\chi(\overline{A}B(\lambda_1,\lambda_2))$ is uniquely determined.
By Proposition~\ref{prop:biobjective_game}(ii), the duality
\[
  \max_{\sigma\in\mathcal{S}}\chi(f^\sigma(\lambda_1,\lambda_2))
  = \chi(\overline{A}B(\lambda_1,\lambda_2))
  = \min_{\tau\in\mathcal{T}}\chi(f^\tau(\lambda_1,\lambda_2))
\]
holds. Taking the componentwise minimum over all $j\in[n+1]$ and
using $\Phi(\lambda_1,\lambda_2)=\min_j\chi_j(\overline{A}B(\lambda_1,
\lambda_2))$, $\Phi^\sigma(\lambda_1,\lambda_2)=\min_j\chi_j(f^\sigma)$,
and $\Phi_\tau(\lambda_1,\lambda_2)=\min_j\chi_j(f^\tau)$, the
equality~\eqref{eq:saddle} follows directly.
\end{proof}

\begin{proposition}
\label{prop:piecewise_linear}
For $A$ and $B(\lambda_1, \lambda_2)$ given by~\eqref{eq:AB_matrices}, 
the functions $\Phi(\lambda_1, \lambda_2)$, $\Phi_\tau(\lambda_1, \lambda_2)$ 
and $\Phi^\sigma(\lambda_1, \lambda_2)$ are componentwise non-decreasing,
piecewise linear, and continuous.
\end{proposition}

\begin{proof}
By Proposition~\ref{prop:biobjective_game}(iii), for any fixed strategy
$\sigma$ of Max, the function
$(\lambda_1,\lambda_2)\mapsto\chi_j(f^\sigma(\lambda_1,\lambda_2))$
is a minimum of finitely many linear functions of $(\lambda_1,\lambda_2)$,
hence piecewise linear, continuous, and componentwise non-decreasing.
Since $\Phi^\sigma(\lambda_1,\lambda_2)=\min_j\chi_j(f^\sigma(\lambda_1,\lambda_2))$
is a minimum of finitely many such functions over $j\in[n+1]$, it
inherits all three properties. By Proposition~\ref{prop:saddle},
$\Phi = \max_\sigma \Phi^\sigma$, which is a maximum of piecewise
linear componentwise non-decreasing continuous functions, preserving
all three properties. The same argument applies to $\Phi_\tau$ using
the symmetric statement for Min strategies.
\end{proof}

\begin{proposition}
\label{prop:reformulation}
For $A$ and $B(\lambda_1, \lambda_2)$ given by~\eqref{eq:AB_matrices},
$\Phi(\lambda_1, \lambda_2) \geq 0$ if and only if there exists a finite 
$z \in \mathbb{R}^{n+1}$ satisfying~\eqref{eq:two_sided}.
\end{proposition}

\begin{proof}
This is a direct application of Proposition~\ref{prop:solvability} to 
the system $A \otimes z \leq B(\lambda_1, \lambda_2) \otimes z$.
\end{proof}

The structure of this mean-payoff game is illustrated in
Figure~\ref{fig:mpg_general}.

\begin{figure}[ht]
\centering
\begin{tikzpicture}[
  minnode/.style={circle,draw=black,fill=white,
    minimum size=0.70cm,inner sep=2pt,thick,font=\small},
  lam1node/.style={rectangle,draw=blue!70!black,fill=blue!8,
    minimum size=0.70cm,inner sep=2pt,thick,font=\small},
  lam2node/.style={rectangle,draw=red!70!black,fill=red!8,
    minimum size=0.70cm,inner sep=2pt,thick,font=\small},
  consnode/.style={rectangle,draw=gray!60!black,fill=gray!8,
    minimum size=0.70cm,inner sep=2pt,thick,font=\small},
  ->,>=Stealth,thick
]
\node[minnode](v1) at (0, 2.8){$v_1$};
\node[minnode](v2) at (0, 1.4){$v_2$};
\node[font=\small]  at (0, 0.6){$\vdots$};
\node[minnode](vn) at (0,-0.2){$v_n$};
\node[minnode](v0) at (0,-2.0){$v_0$};
\node[consnode](c1a) at (-3.6, 2.8){};
\node[consnode](c1b) at (-3.6, 1.4){};
\node[font=\small]   at (-3.6, 0.6){$\vdots$};
\node[consnode](c1m) at (-3.6,-0.2){};
\draw[decorate,decoration={brace,amplitude=5pt},thick]
  (-4.2,-0.6)--(-4.2,3.2) node[midway,left=7pt,font=\small]{$C_1$};
\node[lam1node](c2a) at (3.6, 3.6){};
\node[lam1node](c2b) at (3.6, 2.2){};
\node[font=\small]   at (3.6, 1.4){$\vdots$};
\node[lam1node](c2n) at (3.6, 0.6){};
\draw[decorate,decoration={brace,amplitude=5pt,mirror},thick]
  (4.2,0.2)--(4.2,4.0) node[midway,right=7pt,font=\small]{$C_2$\;($\lambda_1$)};
\node[lam1node](c3) at (3.6,-0.2){$C_3$};
\node[right=0.08cm of c3,font=\scriptsize,blue!80!black]{$(\lambda_1)$};
\node[lam2node](c4a) at (3.6,-1.2){};
\node[lam2node](c4b) at (3.6,-2.3){};
\node[font=\small]   at (3.6,-3.0){$\vdots$};
\node[lam2node](c4n) at (3.6,-3.7){};
\draw[decorate,decoration={brace,amplitude=5pt,mirror},thick]
  (4.2,-4.1)--(4.2,-0.8) node[midway,right=7pt,font=\small]{$C_4$\;($\lambda_2$)};
\node[lam2node](c5) at (3.6,-4.9){$C_5$};
\node[right=0.08cm of c5,font=\scriptsize,red!80!black]{$(\lambda_2)$};
\draw[->](v1)--(c1a) node[midway,above,font=\scriptsize]{$U$};
\draw[->](v2)--(c1b); \draw[->](vn)--(c1m);
\draw[->](v0) to[out=180,in=260](c1m)
  node[near start,below,font=\scriptsize]{$b$};
\draw[->](c1a) to[out=0,in=155](v1)
  node[midway,above,font=\scriptsize]{$V$};
\draw[->](c1b) to[out=0,in=155](v2); \draw[->](c1m) to[out=0,in=205](vn);
\draw[->](c1b) to[out=260,in=145](v0)
  node[near end,left,font=\scriptsize]{$d$};
\draw[->](v1) to[out=25,in=175](c2a)
  node[midway,above,font=\scriptsize]{$p_1$};
\draw[->](v2) to[out=25,in=175](c2b);
\draw[->](vn) to[out=25,in=175](c2n);
\draw[->,dashed,blue!80!black,thick](c2a) to[out=200,in=20](v1)
  node[midway,above,font=\scriptsize,blue!80!black]{$\lambda_1$};
\draw[->,dashed,blue!80!black,thick](c2b) to[out=200,in=20](v2);
\draw[->,dashed,blue!80!black,thick](c2n) to[out=200,in=20](vn);
\draw[->](v1) to[out=335,in=95](c3)
  node[near start,right,font=\scriptsize]{$q_1^-$};
\draw[->](vn) to[out=335,in=95](c3);
\draw[->,dashed,blue!80!black,thick](c3) to[out=255,in=15](v0)
  node[midway,right,font=\scriptsize,blue!80!black]{$\lambda_1$};
\draw[->](v1) to[out=335,in=155](c4a)
  node[near start,font=\scriptsize]{$p_2$};
\draw[->](v2) to[out=335,in=155](c4b);
\draw[->](vn) to[out=335,in=155](c4n);
\draw[->,dashed,red!80!black,thick](c4a) to[out=185,in=345](v1)
  node[midway,below,font=\scriptsize,red!80!black]{$\lambda_2$};
\draw[->,dashed,red!80!black,thick](c4b) to[out=185,in=345](v2);
\draw[->,dashed,red!80!black,thick](c4n) to[out=185,in=345](vn);
\draw[->](v2) to[out=305,in=95](c5)
  node[near start,font=\scriptsize]{$q_2^-$};
\draw[->](vn) to[out=305,in=95](c5);
\draw[->,dashed,red!80!black,thick](c5) to[out=180,in=315](v0)
  node[midway,below,font=\scriptsize,red!80!black]{$\lambda_2$};
\begin{scope}[on background layer]
  \node[fill=gray!10,rounded corners=7pt,fit=(c1a)(c1m),inner sep=6pt]{};
  \node[fill=blue!7,rounded corners=7pt,fit=(c2a)(c2n)(c3),inner sep=6pt]{};
  \node[fill=red!7,rounded corners=7pt,fit=(c4a)(c4n)(c5),inner sep=6pt]{};
  \node[fill=green!7,rounded corners=7pt,fit=(v1)(vn)(v0),inner sep=6pt]{};
\end{scope}
\node[font=\small,blue!80!black] at (3.6, 4.8)
  {$\lambda_1$ group ($C_2\cup C_3$)};
\node[font=\small,red!80!black]  at (3.6,-5.7)
  {$\lambda_2$ group ($C_4\cup C_5$)};
\node[font=\small,gray!70!black] at (-3.6,-1.3){constraint group $C_1$};
\node[font=\small]               at (0,-3.1){Min nodes};
\end{tikzpicture}
\caption{General parametric MPG for the tropical bi-objective
pseudolinear optimization problem. Circles: Min nodes
($v_1,\ldots,v_n$ and free-standing $v_0$). Squares: Max nodes.
The $\lambda_1$ group ($C_2\cup C_3$, blue) and $\lambda_2$ group
($C_4\cup C_5$, red) are disjoint and connect only through the Min
nodes. Min-to-Max arc weights are $-a_{ij}$ (negatives of $A$ entries);
Max-to-Min arc weights are $b_{ij}$ (entries of $B(\lambda_1,\lambda_2)$).
Dashed arcs carry parameter weights $\lambda_1$ (blue) or $\lambda_2$ (red).
The disjointness of the two parameter groups is the
structural property underlying Theorem~\ref{thm:pareto_structure}(iii).}
\label{fig:mpg_general}
\end{figure}
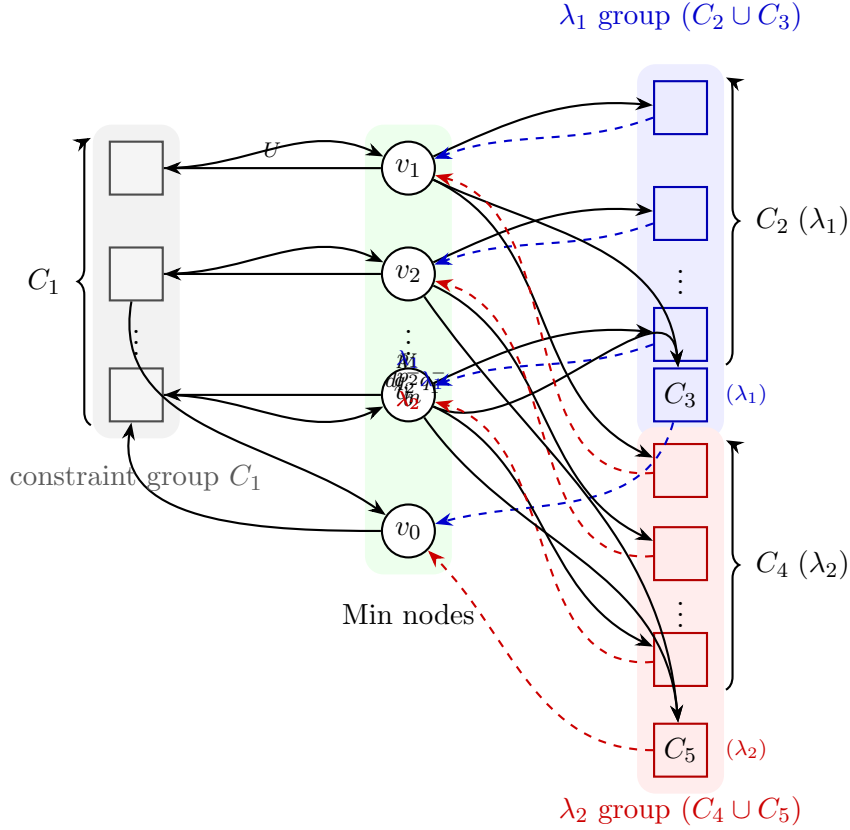

\section{Structure of the Feasibility Region and Pareto Front}
\label{sec:structure}

We now establish the main structural results of this paper.

\begin{lemma}
\label{lem:k_bound}
\label{lem:k1k2_bound}
For any elementary cycle $\gamma$: $k_1(\gamma) \leq 2$,
$k_2(\gamma) \leq 2$, and $k_1(\gamma) + k_2(\gamma) \leq 2$.
\end{lemma}

\begin{proof}
The arc structure used below is established by the matrix
construction in~\eqref{eq:AB_matrices}, every $C_2$-node $L_{1,i}$
has incoming arcs only from Min node $v_0$ (since the $p_1$-rows of
$A$ have finite entries only in the $t$-column, i.e.\ column $v_0$),
and every $C_4$-node $L_{2,i}$ similarly has incoming arcs only from
$v_0$; every $C_3$-node has its unique outgoing arc to $v_0$, and
every $C_5$-node likewise.

Since every $C_2$-node has incoming arcs only from $v_0$, visiting any
$C_2$-node requires $v_0$ as its immediate predecessor. Since $v_0$
appears at most once in an elementary cycle, \emph{at most one}
$C_2$-node can appear. Together with $C_3$ appearing at most once,
$k_1(\gamma) \leq 2$. By symmetry $k_2(\gamma) \leq 2$.

For the joint bound, suppose $k_1(\gamma) = 2$: both some $C_2$-node
$L_{1,j}$ and $C_3$ appear. Since $L_{1,j}$ requires $v_0$ as
predecessor and $C_3$ has outgoing arc only to $v_0$, the cycle
contains the sub-path $v_0 \to L_{1,j} \to v_j \to \cdots \to v_k
\to C_3 \to v_0$, so $v_0$ is visited exactly once, serving as both
the predecessor of $L_{1,j}$ and the successor of $C_3$.
Any $C_4$-node also requires $v_0$ as immediate predecessor; adding it
would require $v_0$ to appear a second time, contradicting elementarity.
Any $C_5$-node has outgoing arc only to $v_0$; adding it would also
create a second visit to $v_0$. Hence $k_2(\gamma) = 0$ when
$k_1(\gamma) = 2$. By symmetry $k_1(\gamma) = 0$ when $k_2(\gamma)=2$.
Therefore $k_1(\gamma) + k_2(\gamma) \leq 2$.
\end{proof}

\begin{theorem}
\label{thm:pareto_structure}
(Structure of the Pareto Front).
Consider the tropical bi-objective pseudolinear optimization 
problem~\eqref{eq:biobjective}--\eqref{eq:constraint}. 
Let $\calR$ and $\calP$ be the feasibility region and Pareto front 
defined in~\eqref{eq:feasibility_region}--\eqref{eq:pareto_front}. Then:

\begin{enumerate}
  \item[\emph{(i)}] $\calR$ is a convex subset of $\mathbb{R}^2$.
  
  \item[\emph{(ii)}] $\calP$ is a convex piecewise-linear curve in 
  $\mathbb{R}^2$ with finitely many breakpoints.
  
  \item[\emph{(iii)}] When the finite entries of 
  $p_1, p_2, q_1, q_2, U, V, b, d$ are all integers, every breakpoint 
  of $\calP$ has coordinates that are integer multiples of $\tfrac{1}{2}$.
\end{enumerate}
\end{theorem}

\begin{proof}
\textbf{(i) Convexity of $\calR$.}

By Proposition~\ref{prop:reformulation}, 
$\calR = \{\Phi(\lambda_1, \lambda_2) \geq 0\}$.
We show that $\calR$ is an intersection of halfspaces in $\mathbb{R}^2$,
hence convex.

By Proposition~\ref{prop:saddle},
\[
  \Phi(\lambda_1, \lambda_2) \geq 0
  \iff
  \min_{\tau \in \mathcal{T}} \Phi_\tau(\lambda_1,\lambda_2) \geq 0
  \iff
  \Phi_\tau(\lambda_1,\lambda_2) \geq 0 \quad \text{for all } \tau \in \mathcal{T}.
\]
Fix any strategy $\tau$ of Min. The condition 
$\Phi_\tau(\lambda_1,\lambda_2) \geq 0$ means that every cycle $\gamma$ 
in the reduced game defined by $A_\tau$ and $B(\lambda_1,\lambda_2)$ 
has non-negative mean weight. By Lemma~\ref{lem:cyclemean}, the mean 
weight of $\gamma$ is
\[
  \mu(\gamma,\lambda_1,\lambda_2)
  = \frac{k_1(\gamma)\,\lambda_1 + k_2(\gamma)\,\lambda_2 + s(\gamma)}{|\gamma|},
\]
where $k_1(\gamma), k_2(\gamma) \geq 0$ count the $\lambda_1$-arcs and
$\lambda_2$-arcs in $\gamma$, and $s(\gamma) \in \mathbb{R}$ is
determined entirely by the data. The non-negativity condition 
$\mu(\gamma,\lambda_1,\lambda_2) \geq 0$ is therefore the closed halfspace
\[
  k_1(\gamma)\,\lambda_1 + k_2(\gamma)\,\lambda_2 \geq -s(\gamma).
\]
Since this must hold for every cycle $\gamma$ and every $\tau \in \mathcal{T}$:
\[
  \calR 
  = \bigcap_{\tau \in \mathcal{T}}\; \bigcap_{\gamma} 
    \Bigl\{(\lambda_1,\lambda_2) \in \mathbb{R}^2 : 
    k_1(\gamma)\,\lambda_1 + k_2(\gamma)\,\lambda_2 \geq -s(\gamma) \Bigr\}.
\]
This is an intersection of closed halfspaces, hence convex.

\medskip
\textbf{(ii) Piecewise linearity and finiteness of breakpoints.}

By Proposition~\ref{prop:piecewise_linear}, the functions $\Phi_\tau$ 
are piecewise linear and continuous for each $\tau \in \mathcal{T}$, 
and $\Phi = \min_\tau \Phi_\tau$ is their minimum. Since $\mathcal{T}$ 
is finite, $\Phi$ is itself piecewise linear and continuous. The 
Pareto front $\calP$, being the lower boundary of $\calR$ where 
$\Phi(\lambda_1,\lambda_2) = 0$, is therefore piecewise linear. 
The number of pieces is bounded by the number of distinct linear pieces 
across all $\Phi_\tau$, which is finite since $|\mathcal{T}|$ is finite. 
Hence $\calP$ has finitely many breakpoints. 

\medskip
\textbf{(iii) Denominator bound for integer data.}

We analyse the arc structure of the matrices $A$ and
$B(\lambda_1,\lambda_2)$ in~\eqref{eq:AB_matrices} to obtain a tight
bound on $k_1(\gamma)$ and $k_2(\gamma)$.

\medskip
\noindent\textit{Outgoing arcs} (from $B$): Each $C_2$-node $L_{1,i}$
has unique outgoing arc $L_{1,i} \to v_i$ (weight $\lambda_1$,
from the $\lambda_1 I$ block); $C_3 = L_{1,0}$ has unique outgoing arc
$L_{1,0} \to v_0$ (weight $\lambda_1$, from the scalar $\lambda_1$ entry).
Each $C_4$-node $L_{2,i}$ has unique outgoing arc $L_{2,i} \to v_i$
(weight $\lambda_2$); $C_5 = L_{2,0}$ has unique outgoing arc
$L_{2,0} \to v_0$ (weight $\lambda_2$).

\medskip
\noindent\textit{Incoming arcs} (from $A$): The $p_1$-rows of $A$ have
$-\infty$ in columns $1,\ldots,n$ and a finite entry only in column $n+1$
(corresponding to $v_0$). Thus every $C_2$-node $L_{1,i}$ has incoming
arcs \emph{only from $v_0$}. The $q_1^-$-row of $A$ has finite entries in
columns $1,\ldots,n$ and $-\infty$ in column $n+1$, so $C_3$ has incoming
arcs \emph{only from $v_1,\ldots,v_n$}. Symmetrically, every $C_4$-node
has incoming arcs only from $v_0$, and $C_5$ has incoming arcs only from
$v_1,\ldots,v_n$.

By Lemma~\ref{lem:k_bound}, the achievable pairs
$(k_1(\gamma), k_2(\gamma))$ lie in
$\{(0,0),(1,0),(0,1),(1,1),(2,0),(0,2)\}$.

A breakpoint of $\calP$ occurs where two strategy-restricted boundary
equations of Max simultaneously equal zero:
\begin{equation}
\label{eq:breakpoint_system}
  k_1^{(1)} \lambda_1 + k_2^{(1)} \lambda_2 + s^{(1)} = 0, \qquad
  k_1^{(2)} \lambda_1 + k_2^{(2)} \lambda_2 + s^{(2)} = 0,
\end{equation}
with $(k_1^{(r)}, k_2^{(r)}) \in \{(0,0),(1,0),(0,1),(1,1),(2,0),(0,2)\}$
and $s^{(r)} \in \mathbb{Z}$.
The determinant $\Delta = k_1^{(1)} k_2^{(2)} - k_1^{(2)} k_2^{(1)}$
is non-zero (the two pieces have distinct slopes).
We consider two cases.

\medskip
\noindent\textbf{Case~1:} $(k_1^{(1)}, k_2^{(1)}) = (2,0)$ and 
$(k_1^{(2)}, k_2^{(2)}) = (0,2)$, or vice versa. Here $|\Delta| = 4$, but the 
system~\eqref{eq:breakpoint_system} decouples: $2\lambda_1 = -s^{(1)}$ and 
$2\lambda_2 = -s^{(2)}$, giving $\lambda_1 = -s^{(1)}/2$ and 
$\lambda_2 = -s^{(2)}/2$ directly. Both coordinates are integer multiples 
of~$\tfrac{1}{2}$.

\medskip
\noindent\textbf{Case~2:} All remaining pairs with $\Delta \neq 0$. 
Direct enumeration over all such pairs from the set above gives 
$|\Delta| \in \{1, 2\}$. By Cramer's rule, each breakpoint coordinate 
has denominator dividing $|\Delta| \leq 2$, so again every coordinate 
is an integer multiple of~$\tfrac{1}{2}$.

\medskip
\noindent In both cases every breakpoint coordinate is an integer 
multiple of~$\tfrac{1}{2}$. 
\end{proof}


\begin{corollary}
\label{cor:solution_map}
\textbf{(Optimal Solution Map).}
For each $(\lambda_1, \lambda_2) \in \calP$, the set of optimal solutions 
$x^*(\lambda_1, \lambda_2)$ is non-empty and varies piecewise linearly 
along $\calP$.
\end{corollary}

\begin{proof}
Existence follows from Proposition~\ref{prop:reformulation} and the 
definition of $\calP$. Piecewise linearity follows from the fact that on 
each linear piece of $\calP$, the active strategy $\sigma$ of Max is fixed, 
and the solution is determined by the alcoved polyhedron of~\cite{PSW2023}, 
whose extreme points depend linearly on $(\lambda_1, \lambda_2)$.
\end{proof}

\section{Optimality and Infeasibility Certificates}
\label{sec:certificates}

We now provide certificates for Pareto optimality and infeasibility
in terms of the mean-payoff game structure.
The optimality certificate (Proposition~\ref{prop:optimality}) is the
bi-objective analogue of the optimality certificate
of~\cite[Proposition~3.5]{PSW2023}. The infeasibility certificate
(Proposition~\ref{prop:infeasibility}) replaces the unboundedness
certificate of~\cite[Proposition~3.6]{PSW2023}, the reason the second
certificate changes both name and meaning in the bi-objective setting
is explained in Remark~\ref{rem:no_unboundedness} below.

\begin{proposition}
\label{prop:optimality}
(Optimality Certificate).
A point $(\lambda_1^*, \lambda_2^*) \in \mathbb{R}^2$ lies on the 
Pareto front $\calP$ if and only if the following two conditions hold:

\begin{enumerate}
  \item[\emph{(i)}] $\Phi(\lambda_1^*, \lambda_2^*) \geq 0$.

  \item[\emph{(ii)}] There exists a strategy $\tau$ of Min such that 
  in the mean-payoff game defined by $A_\tau$ and $B(\lambda_1^*, \lambda_2^*)$, 
  all cycles accessible from some node $j$ of Min have non-positive mean 
  weight, and at least one zero-weight cycle accessible from $j$ contains 
  a Max node from the groups corresponding to $p_1$, $q_1$, $p_2$, or 
  $q_2$ in~\eqref{eq:AB_matrices}.
\end{enumerate}
\end{proposition}

\begin{proof}
$(\lambda_1^*, \lambda_2^*) \in \calP$ means $(\lambda_1^*, \lambda_2^*)
\in \calR$ and $(\lambda_1^*, \lambda_2^*)$ lies on the lower boundary
of $\calR$.

The condition $(\lambda_1^*, \lambda_2^*) \in \calR$ is equivalent to 
condition~(i) by Proposition~\ref{prop:reformulation}.

For the lower boundary condition, we use the halfspace representation
of $\calR$ from Theorem~\ref{thm:pareto_structure}(i) and
Lemma~\ref{lem:cyclemean}. By that lemma, each cycle $\gamma$ 
contributes the halfspace 
$k_1(\gamma)\lambda_1 + k_2(\gamma)\lambda_2 \geq -s(\gamma)$.
The point $(\lambda_1^*, \lambda_2^*)$ lies on the lower boundary
of $\calR$ if and only if at least one such halfspace is active with 
$(k_1(\gamma), k_2(\gamma)) \neq (0,0)$, meaning that a componentwise
decrease in $(\lambda_1^*,\lambda_2^*)$ would exit $\calR$. An active halfspace with 
$k_1(\gamma)\lambda_1^* + k_2(\gamma)\lambda_2^* + s(\gamma) = 0$ and 
$(k_1(\gamma), k_2(\gamma)) \neq (0,0)$ corresponds precisely to a cycle 
$\gamma$ of zero mean weight that carries at least one $\lambda_1$ or 
$\lambda_2$ arc.

Arcs carrying $\lambda_1$ or $\lambda_2$ emanate from Max nodes 
corresponding to the $p_1$, $q_1$, $p_2$, $q_2$ groups 
in~\eqref{eq:AB_matrices}, that is, Max nodes outside the constraint 
group $[m]$. Thus condition~(ii) is precisely the requirement that such 
an active zero-weight cycle exists. The strategy $\tau$ of Min is any 
strategy under which this cycle is accessible from some Min node $j$, 
ensuring $\Phi_\tau(\lambda_1^*, \lambda_2^*) = 0$ with the cycle 
structure described.
\end{proof}

\begin{remark}
\label{rem:optimality_geometric}
Condition~(ii) has a clear geometric interpretation. The Max nodes 
outside the constraint group $[m]$ are precisely those whose outgoing 
arcs carry the parameters $\lambda_1$ or $\lambda_2$. A zero-weight 
cycle through such a node means the current parameter values 
$(\lambda_1^*, \lambda_2^*)$ are exactly tight, reducing either 
$\lambda_1$ or $\lambda_2$ (or both) would make this cycle have 
negative mean weight, pushing the point out of $\calR$. This is why 
such a cycle witnesses Pareto optimality.
\end{remark}

\begin{proposition}
\label{prop:infeasibility}
(Infeasibility Certificate).
The constraint $U \otimes x \oplus b \leq V \otimes x \oplus d$ is 
infeasible, equivalently $\calR = \emptyset$, if and only if 
$\chi_j(\overline{A}_c\,B_c) < 0$ for all $j = 1, \ldots, n+1$,
where $A_c = (U\ b)$ and $B_c = (V\ d)$ are the $m \times (n+1)$
matrices formed by the constraint rows of $A$ and $B(\lambda_1,\lambda_2)$
respectively.
\end{proposition}

\begin{proof}
The feasibility region $\calR$ is empty if and only if there exists no 
finite $x \in \mathbb{R}^n$ satisfying 
$U \otimes x \oplus b \leq V \otimes x \oplus d$. This is because 
for any finite feasible $x$, taking $\lambda_k = f_k(x)$ gives a 
point $(\lambda_1, \lambda_2) \in \calR$.

The structural constraint $U \otimes x \oplus b \leq V \otimes x 
\oplus d$ is precisely the two-sided system $A_c \otimes z \leq 
B_c \otimes z$ where $z = (x^\top, t)^\top \in \mathbb{R}^{n+1}$
and the rows of $A_c$ and $B_c$ are the $m$ constraint rows of 
$A$ and $B(\lambda_1,\lambda_2)$ (which are independent of $\lambda_1,
\lambda_2$: $A_c = (U\ b)$ and $B_c = (V\ d)$ at all parameter values).

By Proposition~\ref{prop:solvability} applied to the matrices $A_c$ 
and $B_c$:
\[
  \chi_j(\overline{A}_c\,B_c) \geq 0
  \;\iff\;
  \text{there exists } z \in \mathbb{R}^{n+1} \text{ with }
  A_c \otimes z \leq B_c \otimes z \text{ and } z_j \neq -\infty.
\]
The right-hand condition (with $j \in \{1,\ldots,n\}$, so $z_j = x_j \neq -\infty$) 
is exactly the feasibility of $U \otimes x \oplus b \leq V \otimes x \oplus d$
with $x_j$ finite. Therefore $\calR \neq \emptyset$ if and only if 
$\chi_j(\overline{A}_c\,B_c) \geq 0$ for some $j$, and 
$\calR = \emptyset$ if and only if $\chi_j(\overline{A}_c\,B_c) < 0$ 
for all $j = 1, \ldots, n+1$.
\end{proof}

\begin{remark}
\label{rem:infeasibility_practical}
In practice, infeasibility can be detected before running the bisection 
algorithm by calling an MPG solver on the constraint subgame defined 
by $A_c = (U\ b)$ and $B_c = (V\ d)$. If $\chi_j(\overline{A}_c\,B_c) < 0$ 
for all $j$, the structural constraint is infeasible and the algorithm 
terminates immediately. This check adds no asymptotic overhead since 
the MPG solve is already required in the bisection algorithm.
\end{remark}

\begin{remark}
\label{rem:no_unboundedness}
Parsons et~al. \cite{PSW2023} also establish exactly two
certificates: optimality and unboundedness, mirroring the
classical LP duality framework. Their \emph{unboundedness certificate}
applies when the scalar parameter $\lambda$ has no finite minimum:
the constraint $A \otimes x \leq B(\lambda) \otimes x$ remains feasible
for every $\lambda \in \mathbb{R}$, so the infimum of $\lambda$ is
$-\infty$. This occurs precisely when no $\lambda$-arc participates in
any critical cycle of the game, so reducing $\lambda$ never causes
infeasibility. The certificate is therefore a statement about the
\emph{cycle structure} of the game relative to the parameter arcs.

In the present bi-objective setting. We are not minimizing a single scalar $\lambda$ over
$\mathbb{R}$, we are characterizing the lower boundary $\calP$ of
the feasibility region $\calR \subseteq \mathbb{R}^2$. There is no
notion of pushing a parameter to $-\infty$, so unboundedness in the sense of Parsons et~al. \cite{PSW2023} simply does not arise. Instead, the
relevant second outcome is that the constraint set
$U \otimes x \oplus b \leq V \otimes x \oplus d$ is itself empty, no finite $x$ satisfies the constraints regardless of the parameter
values $(\lambda_1,\lambda_2)$. This is genuine \emph{infeasibility},
and Proposition~\ref{prop:infeasibility} certifies it via the
constraint subgame with matrices $A_c$ and $B_c$. When $\calR \neq \emptyset$,
the Pareto front $\calP$ always exists as a well-defined convex
piecewise-linear curve, even if it extends to infinity as a set, an
unbounded $\calP$ is ordinary geometry already covered by
Theorem~\ref{thm:pareto_structure}, not a degenerate case requiring
a separate certificate. The two-certificate framework,
Propositions~\ref{prop:optimality} and~\ref{prop:infeasibility},
is therefore complete for the bi-objective problem.
\end{remark}

\subsection*{Verification on the Worked Examples}

\textit{Example~1.} At $P_2=(1,1)$, the cycle
$\gamma:L_1^{(1)}\!\to v_1\!\to L_2^{(0)}\!\to v_0\!\to L_1^{(1)}$
has $(k_1,k_2,s)=(1,1,-2)$ and $W_\gamma(1,1)=1+1-2=0$, witnessing
Pareto optimality.

\textit{Example~2.} At $B^*=(1,1)$, two cycles are simultaneously
active: the Piece~1 cycle has $(k_1,k_2,s)=(1,0,-1)$,
$W_\gamma(1,1)=1-1=0$; the Piece~2 cycle has $(k_1,k_2,s)=(1,1,-2)$,
$W_\gamma(1,1)=1+1-2=0$.  Their simultaneous zero weight confirms $B^*$
as the breakpoint.  At $P_4=(2,0)$, $(k_1,k_2,s)=(1,0,-2)$,
$W_\gamma(2,0)=2-2=0$.

\section{Directional Bisection Algorithm}
\label{sec:algorithm}
\subsection{Tracing a Single Pareto Point}

To find one point on the Pareto front, we fix a weight 
$w \in (0,1)$ and solve the weighted scalar problem:

\begin{equation}
\label{eq:weighted_scalar}
  \lambda^*(w) = \min\left\{ w\lambda_1 + (1-w)\lambda_2 : 
  \Phi(\lambda_1, \lambda_2) \geq 0 \right\}.
\end{equation}

This is equivalent to bisecting on the parameter 
$c = w\lambda_1 + (1-w)\lambda_2$ along the family of lines 
$\{w\lambda_1 + (1-w)\lambda_2 = c\}$. For each value of $c$, the 
feasibility condition $\Phi(\lambda_1,\lambda_2) \geq 0$ on the line 
reduces to a single-parameter problem that can be solved by the 
bisection method of~\cite{PSW2023}.

Specifically, substituting $\lambda_2 = (c - w\lambda_1)/(1-w)$ into 
the system $A \otimes z \leq B(\lambda_1, \lambda_2) \otimes z$, 
the parameter $c$ plays the role of the single scalar parameter 
$\lambda$ in~\cite{PSW2023}, and bisection on $c$ gives the optimal 
value $c^*$. The corresponding Pareto point is then 
$(\lambda_1^*, \lambda_2^*)$ with $w\lambda_1^* + (1-w)\lambda_2^* = c^*$.

\subsection{Directional Scalarisation and Full Front Recovery}

By Theorem~\ref{thm:pareto_structure}(ii), the Pareto front $\calP$ has 
finitely many breakpoints. Each linear piece of $\calP$ corresponds to a 
direction $w \in (0,1)$ that is perpendicular to that piece. By sweeping 
$w$ over a finite grid of values, we recover all pieces of $\calP$.

For integer data, by Theorem~\ref{thm:pareto_structure}(iii), all 
breakpoints have coordinates that are integer multiples of $\tfrac{1}{2}$. 
The number of distinct breakpoints is bounded by the number of strategies 
of Min in the MPG, which is finite. Therefore the algorithm below 
terminates in a finite number of steps.

\begin{algorithm2e}[H]
\caption{Directional Bisection for the Pareto Front}
\label{alg:directional_bisection}
\DontPrintSemicolon
\KwIn{Data matrices $U,V,b,d,p_1,p_2$ with integer or $-\infty$
  entries; $q_1,q_2$ with integer or $+\infty$ entries;
  weight grid $W=\{w_1,\ldots,w_K\}\subset(0,1)$.}
\KwOut{Pareto-optimal triples
  $\{(\lambda_1^{(k)},\lambda_2^{(k)},x^{(k)})\}_{k=1}^{K}$
  with $(\lambda_1^{(k)},\lambda_2^{(k)})\in\calP$ and
  $f_j(x^{(k)})=\lambda_j^{(k)}$ for $j=1,2$.}
\BlankLine
\For{$k = 1$ \KwTo $K$}{
  \BlankLine
  \tcp{Initialise bounds for direction $w_k$}
  $c^{(-)} \leftarrow (q_1^-\otimes(p_1\oplus p_2))^{\otimes 1/2}$
  \Comment*[r]{lower bound: \cite[Lemma~3.2]{PSW2023}}\;
  Find feasible $x$ for $U\otimes x\oplus b\leq V\otimes x\oplus d$\;
  $c^{(+)} \leftarrow \lceil w_k f_1(x)+(1-w_k)f_2(x)\rceil_{1/2}$
  \Comment*[r]{upper bound}
  \BlankLine
  \tcp{Bisection on the scalar $c$}
  \While{$c^{(-)} < c^{(+)}$}{
    $c \leftarrow \lfloor (c^{(-)}+c^{(+)})/2 \rceil_{1/2}$
    \Comment*[r]{midpoint rounded to half-integer}
    Substitute $\lambda_2 \leftarrow (c - w_k\lambda_1)/(1-w_k)$ into
      $A\otimes z \leq B(\lambda_1,\lambda_2)\otimes z$\;
    Solve the resulting single-parameter system in $\lambda_1$
      via~\cite{PSW2023} to obtain $(\lambda_1,\lambda_2)$\;
    \eIf{a finite solution $z$ exists}{
      $c^{(+)} \leftarrow c$
      \Comment*[r]{feasible: tighten upper bound}
    }{
      $c^{(-)} \leftarrow c$
      \Comment*[r]{infeasible: raise lower bound}
    }
  }
  \BlankLine
  \tcp{Record Pareto triple}
  $(\lambda_1^{(k)},\lambda_2^{(k)}) \leftarrow$ point on
    $w_k\lambda_1+(1-w_k)\lambda_2=c^{(-)}$\;
  Solve $A\otimes z\leq B(\lambda_1^{(k)},\lambda_2^{(k)})\otimes z$
    and set $x^{(k)}\leftarrow(z_1,\ldots,z_n)^\top$\;
}
\BlankLine
\KwRet{$\{(\lambda_1^{(k)},\lambda_2^{(k)},x^{(k)})\}_{k=1}^{K}$}
\end{algorithm2e}

\begin{proposition}
\label{prop:termination}
Let the finite data entries be integers bounded in absolute value by $M$.
For each direction $w_k$, the bisection sub-loop of
Algorithm~\ref{alg:directional_bisection} terminates in at most
$\lceil\log_2(2M)\rceil + 1$ oracle calls.
Each oracle call constructs the shadow graph and applies Karp's
minimum-cycle-mean algorithm~\cite{Karp1978}, costing
$O\bigl(n^2(n+m)\bigr)$ arithmetic operations.
The total cost of Algorithm~\ref{alg:directional_bisection} per
direction is therefore
\[
  O\!\left(n^2(n+m)\log M\right).
\]
For $m = O(n)$ this simplifies to $O(n^3 \log M)$.
The overall complexity is pseudo-polynomial: it is polynomial in $n$,
$m$, and $M$, but exponential in the bit-length of the data.
\end{proposition}

\begin{proof}
The bisection sub-loop operates on the interval $[c^{(-)}, c^{(+)}]$
and halves the search interval at each step by rounding the midpoint
to the nearest half-integer.
The lower bound $c^{(-)} = (q_1^-\otimes(p_1\oplus p_2))^{\otimes 1/2}$
is a valid lower bound on the optimal value $c^*$ because for any
feasible $x$ and any $(\lambda_1,\lambda_2)\in\calR$, one has
$w\lambda_1+(1-w)\lambda_2 \geq c^{(-)}$; see~\cite[Lemma~3.2]{PSW2023}
for the derivation in the single-parameter setting, which carries
over directly.
Since the data is integer and bounded by $M$, the initial interval
satisfies $c^{(+)} - c^{(-)} \leq 2M$, so there are at most
$4M + 1$ candidate half-integer values in the interval.
By Theorem~\ref{thm:pareto_structure}(iii), the optimal value $c^*$ is
a half-integer, so bisection on half-integers terminates in at most
$\lceil\log_2(4M+1)\rceil \leq \lceil\log_2(2M)\rceil + 1$ steps.

Each oracle call requires:
(i) constructing the shadow graph $W$ with $n+1$ Min nodes, where each
of the $O(n^2)$ entries $W[j][k] = \max_r(B[r][k]-A[r][j])$ is
computed by scanning the $m+2n+2$ rows of $A$ and $B$, costing
$O\bigl((n+1)^2(m+2n+2)\bigr) = O\bigl(n^2(n+m)\bigr)$; and
(ii) running Karp's algorithm on the shadow graph with $n+1$ nodes,
costing $O((n+1)^2) = O(n^2)$.
The dominant cost is step (i), giving $O\bigl(n^2(n+m)\bigr)$
per oracle call and $O\bigl(n^2(n+m)\log M\bigr)$ per direction.
\end{proof}

\section{Scheduling Problem}
\label{sec:example}

We illustrate the theory with two complete examples. We first describe
their common scheduling interpretation, then work through each in full.

\subsection{Scheduling Interpretation}
\label{sec:scheduling}

Both examples model two-task project scheduling problems ($n=2$,
$x=(x_1,x_2)^\top\in\mathbb{R}^2$) with a flow-time type objective
$f_1$ and a lateness type objective $f_2$.  The Pareto front $\calP$
gives the complete trade-off: for each achievable level of $f_1$, the
minimum $f_2$ simultaneously achievable, and vice versa.

\medskip
\noindent\textbf{Example~1 — Precedence-constrained schedule.}
The constraint $|x_1 - x_2| \leq 1$ encodes a two-way precedence
relation, neither task can start more than one time unit after the
other. This models two tasks that must be carried out in near-parallel,
such as two stages of a pipeline that cannot be too far out of
synchronisation. The objective $f_1(x) = \max(2 - x_1,\, x_2)$
penalises scheduling task~1 too late (the term $2 - x_1$ is positive
when $x_1 < 2$) and task~2 too early (the term $x_2$ grows with the
start time of task~2). The objective $f_2(x) = \max(-x_2,\, x_1)$
penalises task~2 starting before time~0 and task~1 starting late.
The Pareto front $\lambda_1 + \lambda_2 = 2$ for $\lambda_1 \in
[\tfrac{1}{2}, 2]$ shows that the total penalty is always~$2$,
reducing one penalty forces an equal increase in the other, giving a
perfectly symmetric trade-off.

\medskip
\noindent\textbf{Example~2 — Release date and deadline constraints.}
The constraint $x_1 \leq 1$ encodes a deadline: task~1 must start by
time~$1$. The constraint $x_2 \geq 0$ encodes a release date: task~2
cannot start before time~$0$. These are the most basic feasibility
constraints in scheduling. The objective $f_1(x) = \max(2 - x_1,\, x_2)$
again penalises task~1 starting too late and task~2 starting too early.
The objective $f_2(x) = \max(1 - x_2,\, x_1)$ penalises task~2
starting after its release time window and task~1 starting late.

The two-piece Pareto front reflects the asymmetry introduced by the
hard deadline on task~1. On Piece~1 ($\lambda_1 = 1$, $\lambda_2 \geq
1$), the deadline $x_1 \leq 1$ is binding: task~1 cannot start earlier,
so $f_1$ is pinned at its minimum value of~$1$ while $f_2$ can be
reduced further by delaying task~2. On Piece~2 ($\lambda_1 + \lambda_2
= 2$, $\lambda_2 \leq 1$), the deadline is no longer the bottleneck and
both penalties trade off symmetrically. The breakpoint $B^* = (1,1)$
is the scheduling threshold: at $(1,1)$, the deadline on task~1 and the
joint penalty bound $\lambda_1 + \lambda_2 = 2$ are simultaneously
active, marking the transition between the two regimes.

\subsection{Example 1: A Single-Piece Pareto Front}
\label{sec:example1}

\subsection{Problem Data}

Let $n = 2$, $m = 2$, and take:

\[
  U = \begin{pmatrix} 0 & -\infty \\ -\infty & 0 \end{pmatrix}, \quad
  V = \begin{pmatrix} -\infty & 1 \\ 1 & -\infty \end{pmatrix}, \quad
  b = \begin{pmatrix} -\infty \\ -\infty \end{pmatrix}, \quad
  d = \begin{pmatrix} -\infty \\ -2 \end{pmatrix},
\]
\[
  p_1 = \begin{pmatrix} 2 \\ -\infty \end{pmatrix}, \quad
  p_2 = \begin{pmatrix} -\infty \\ 0 \end{pmatrix}, \quad
  q_1 = \begin{pmatrix} +\infty \\ 0 \end{pmatrix}, \quad
  q_2 = \begin{pmatrix} 0 \\ +\infty \end{pmatrix}.
\]

\subsection{The Constraint Set}

Expanding $U\otimes x\oplus b\le V\otimes x\oplus d$ row by row:
Row~1 gives $x_1\le 1+x_2$ and Row~2 gives $x_2\le 1+x_1$.
Together these are equivalent to $|x_1-x_2|\le 1$, giving the feasible
region $\mathcal{R}=\{x\in\mathbb{R}^2:|x_1-x_2|\le 1\}$.

\subsection{The Objective Functions}

Computing directly from the definitions:
\[
  f_1(x) = x^- \otimes p_1 \oplus q_1^- \otimes x 
         = \max(2 - x_1,\ x_2),
\]
\[
  f_2(x) = x^- \otimes p_2 \oplus q_2^- \otimes x 
         = \max(-x_2,\ x_1).
\]

\subsection{The Matrices $A$ and $B(\lambda_1, \lambda_2)$}

Following~\eqref{eq:AB_matrices} with $z = (x_1, x_2, t)^\top \in \mathbb{R}^3$,
the five row groups give matrices of size $(m+2n+2)\times(n+1) = 8\times 3$:

\[
  A = \begin{pmatrix}
    0      & -\infty & -\infty \\   
    -\infty & 0      & -\infty \\   
    -\infty & -\infty & 2      \\   
    -\infty & -\infty & -\infty \\  
    -\infty & 0      & -\infty \\   
    -\infty & -\infty & -\infty \\  
    -\infty & -\infty & 0      \\   
    0      & -\infty & -\infty      
  \end{pmatrix},
\]
\[
  B(\lambda_1, \lambda_2) = \begin{pmatrix}
    -\infty & 1      & -\infty \\   
    1      & -\infty & -2     \\    
    \lambda_1 & -\infty & -\infty \\ 
    -\infty & \lambda_1 & -\infty \\ 
    -\infty & -\infty & \lambda_1 \\ 
    \lambda_2 & -\infty & -\infty \\ 
    -\infty & \lambda_2 & -\infty \\ 
    -\infty & -\infty & \lambda_2    
  \end{pmatrix}.
\]

\subsection{The Feasibility Region and Pareto Front}

For a given $(\lambda_1, \lambda_2) \in \mathbb{R}^2$, the system 
$f_1(x) \leq \lambda_1$, $f_2(x) \leq \lambda_2$, $x \in \mathcal{F}$ 
is feasible if and only if there exists $x$ with:
\begin{alignat*}{2}
  2 - \lambda_1 &\leq x_1 \leq \lambda_2, &\qquad& 
  \text{(from } f_1 \leq \lambda_1 \text{ and } f_2 \leq \lambda_2\text{)},\\
  -\lambda_2 &\leq x_2 \leq \lambda_1, &\qquad& 
  \text{(from } f_1 \leq \lambda_1 \text{ and } f_2 \leq \lambda_2\text{)},\\
  |x_1 - x_2| &\leq 1. &\qquad& \text{(feasibility constraint)}
\end{alignat*}

Analysing these conditions, the feasibility region is:
\[
  \calR = \left\{ (\lambda_1, \lambda_2) \in \mathbb{R}^2 : 
  \lambda_1 + \lambda_2 \geq 2,\ \lambda_1 \geq \tfrac{1}{2},\ 
  \lambda_2 \geq 0 \right\}.
\]

The Pareto front is the lower boundary of $\calR$:
\[
  \calP = \left\{ (\lambda_1, \lambda_2) : 
  \lambda_1 + \lambda_2 = 2,\ 
  \tfrac{1}{2} \leq \lambda_1 \leq 2 \right\}.
\]

This is a single linear piece, consistent with Theorem~\ref{thm:pareto_structure}(ii). 
All boundary values $\tfrac{1}{2}$ and $2$ are integer multiples of $\tfrac{1}{2}$, 
consistent with Theorem~\ref{thm:pareto_structure}(iii).

\subsection{Summary and Verification}

Three points on $\calP$ are verified: at each, the optimal $x^*$
satisfies $|x_1-x_2|\le 1$ and $f_k(x^*)=\lambda_k$ for $k=1,2$;
reducing either objective strictly forces the other to worsen.
Results are in Table~\ref{tab:pareto_points}.

\begin{table}[h]
\centering
\caption{Verified Pareto points for the worked example.}
\label{tab:pareto_points}
\begin{tabular}{ccccc}
\toprule
Point & $(\lambda_1, \lambda_2)$ & Optimal $x^*$ & $f_1(x^*)$ & $f_2(x^*)$ \\
\midrule
$P_1$ & $(1/2,\ 3/2)$ & $(3/2,\ 1/2)^\top$ & $1/2$ & $3/2$ \\
$P_2$ & $(1,\ 1)$ & $(1,\ 0)^\top$ & $1$ & $1$ \\
$P_3$ & $(3/2,\ 1/2)$ & $(1/2,\ -1/2)^\top$ & $3/2$ & $1/2$ \\
\bottomrule
\end{tabular}
\end{table}

\noindent All three points satisfy $\lambda_1 + \lambda_2 = 2$, confirming 
the Pareto front. The optimal solution map along $\calP$ is:
\[
  x^*(\lambda_1, \lambda_2) = \begin{pmatrix} \lambda_2 \\ \lambda_1 - 1 \end{pmatrix}, 
  \quad (\lambda_1, \lambda_2) \in \calP,
\]
which varies linearly along the Pareto front as predicted by 
Corollary~\ref{cor:solution_map}.

Figure~\ref{fig:pareto_example1} shows the feasibility region $\calR$
and the Pareto front $\calP$ geometrically, together with three
directional bisection families illustrating how different weight
choices $w$ recover different points on $\calP$.

\begin{figure}[ht]
\centering
\begin{tikzpicture}[scale=2.2, >=Stealth]
  \draw[->](-0.15,0)--(2.6,0) node[right]{$\lambda_1$};
  \draw[->](0,-0.15)--(0,2.6) node[above]{$\lambda_2$};
  \foreach \x in {0.5,1,1.5,2,2.5}
    \draw(\x,0.03)--(\x,-0.03) node[below,font=\scriptsize]{$\x$};
  \foreach \y in {0.5,1,1.5,2,2.5}
    \draw(0.03,\y)--(-0.03,\y) node[left,font=\scriptsize]{$\y$};
  \fill[blue!12,opacity=0.8]
    (0.5,1.5)--(2,0)--(2.5,0)--(2.5,2.5)--(0.5,2.5)--cycle;
  \draw[gray,thick,dashed](0.5,0)--(0.5,2.5)
    node[above,font=\scriptsize,gray]{$\lambda_1=\tfrac{1}{2}$};
  \draw[red,ultra thick](0.5,1.5)--(2,0)
    node[midway,above right,font=\small,red]
    {$\mathcal{P}:\lambda_1+\lambda_2=2$};
  \draw[->,thick,red!60!black](1.25,0.75)--(1.05,0.55)
    node[below left,font=\scriptsize]{SW};
  \filldraw[red](0.5,1.5) circle(1.2pt)
    node[left,font=\scriptsize]{$(\tfrac{1}{2},\tfrac{3}{2})$};
  \filldraw[red](1,1) circle(1.2pt)
    node[above right,font=\scriptsize]{$(1,1)$};
  \filldraw[red](1.5,0.5) circle(1.2pt)
    node[above right,font=\scriptsize]{$(\tfrac{3}{2},\tfrac{1}{2})$};
  \filldraw[red](2,0) circle(1.2pt)
    node[below right,font=\scriptsize]{$(2,0)$};
  \draw[green!60!black,thick,dotted](0,1.75)--(2.5,1.75-2.5/3)
    node[right,font=\scriptsize,green!60!black]{$w=\tfrac{1}{4}$};
  \draw[orange,thick,dotted](0,2)--(2,0)
    node[right,font=\scriptsize,orange]{$w=\tfrac{1}{2}$};
  \draw[purple,thick,dotted](1,2)--(5/3,0)
    node[below,font=\scriptsize,purple]{$w=\tfrac{3}{4}$};
  \node[blue!50!black,font=\small] at (1.9,1.9){$\mathcal{R}$};
\end{tikzpicture}
\caption{Feasibility region $\mathcal{R}$ (blue shading) and Pareto
front $\mathcal{P}$ (red segment) for Example~1. The region $\mathcal{R}$
is convex, being bounded by three halfspaces. The Pareto front is the
southwest boundary, the single linear segment $\lambda_1+\lambda_2=2$
for $\lambda_1\in[\tfrac{1}{2},2]$. Dotted lines show the directional
bisection families for three weights: the green family ($w=\tfrac{1}{4}$)
finds the right endpoint $(2,0)$ with $c^*=\tfrac{1}{2}$; the orange
family ($w=\tfrac{1}{2}$) finds the midpoint $(1,1)$ with $c^*=1$; the
purple family ($w=\tfrac{3}{4}$) also converges to $(1,1)$ with $c^*=1$,
because the true weighted sum at $(\tfrac{1}{2},\tfrac{3}{2})$ equals
$\tfrac{3}{4}\cdot\tfrac{1}{2}+\tfrac{1}{4}\cdot\tfrac{3}{2}=\tfrac{3}{4}$,
which is not a half-integer; bisection returns the smallest feasible
half-integer $c^*=1$ above $\tfrac{3}{4}$.
Each family consists of parallel lines $w\lambda_1+(1-w)\lambda_2=c$;
bisection finds the smallest half-integer $c$ at which such a line
touches $\mathcal{R}$.}
\label{fig:pareto_example1}
\end{figure}
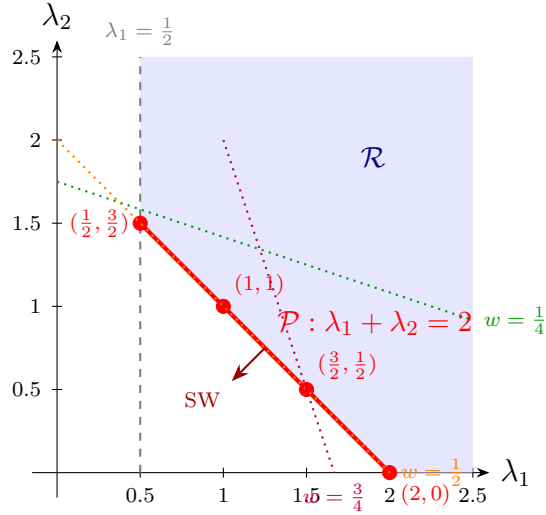

Algorithm~\ref{alg:directional_bisection} applied to Example~1 with
weight grid $W=\{\tfrac{1}{4},\tfrac{1}{2},\tfrac{3}{4}\}$ and
initial feasible point $x=(0,0)^\top$ returns the following
Pareto-optimal triples, each verified to satisfy $f_j(x^{(k)})=\lambda_j^{(k)}$
and the feasibility constraint $|x_1-x_2|\leq 1$:

\begin{center}
\renewcommand{\arraystretch}{1.3}
\begin{tabular}{ccccc}
\toprule
$k$ & $w_k$ & $\lambda_1^{(k)}$ & $\lambda_2^{(k)}$ & $x^{(k)}$ \\
\midrule
$1$ & $\tfrac{1}{4}$ & $2$ & $0$ & $(0,\ 1)^\top$ \\
$2$ & $\tfrac{1}{2}$ & $1$ & $1$ & $(1,\ 0)^\top$ \\
$3$ & $\tfrac{3}{4}$ & $1$ & $1$ & $(1,\ 0)^\top$ \\
\bottomrule
\end{tabular}
\end{center}

\noindent All triples lie on $\lambda_1+\lambda_2=2$,
$\lambda_1\in[\tfrac{1}{2},2]$, confirming
Theorem~\ref{thm:pareto_structure}.  The optimal solution map
$x^*(\lambda_1,\lambda_2)=(\lambda_2,\,\lambda_1-1)^\top$ is
recovered correctly at each point.  Weights $w=\tfrac{1}{2}$ and
$w=\tfrac{3}{4}$ return the same point $(1,1)$: the minimum
$\tfrac{3}{4}\cdot\tfrac{1}{2}+\tfrac{1}{4}\cdot\tfrac{3}{2}=\tfrac{3}{4}$
is not a half-integer, so bisection returns the nearest feasible
half-integer $c^*=1$.

\subsection{Example 2: A Two-Piece Pareto Front}
\label{sec:example2}

We now give a second example where the Pareto front has two distinct 
linear pieces meeting at an interior breakpoint, demonstrating that the 
piecewise-linear structure of Theorem~\ref{thm:pareto_structure}(ii) 
is non-trivial.

\subsubsection*{Problem Data}

Let $n = 2$, $m = 2$, and take:
\[
  U = \begin{pmatrix} 0 & -\infty \\ -\infty & -\infty \end{pmatrix}, \quad
  V = \begin{pmatrix} -\infty & -\infty \\ -\infty & 0 \end{pmatrix}, \quad
  b = \begin{pmatrix} -\infty \\ 0 \end{pmatrix}, \quad
  d = \begin{pmatrix} 1 \\ -\infty \end{pmatrix},
\]
\[
  p_1 = \begin{pmatrix} 2 \\ -\infty \end{pmatrix}, \quad
  p_2 = \begin{pmatrix} -\infty \\ 1 \end{pmatrix}, \quad
  q_1 = \begin{pmatrix} +\infty \\ 0 \end{pmatrix}, \quad
  q_2 = \begin{pmatrix} 0 \\ +\infty \end{pmatrix}.
\]

\subsubsection*{The Constraint Set}

Expanding row by row: Row~1 gives $x_1\le 1$; Row~2 gives $x_2\ge 0$.
The feasible region is $\mathcal{R}=\{x\in\mathbb{R}^2:x_1\le 1,\,x_2\ge 0\}$.

\subsubsection*{The Objective Functions}
\[
  f_1(x)=\max(2-x_1,\,x_2),\qquad f_2(x)=\max(1-x_2,\,x_1).
\]

The corresponding parametric MPG is shown in
Figure~\ref{fig:mpg_example2}.

\begin{figure}[ht]
\centering
\begin{tikzpicture}[
  minnode/.style={circle,draw=black,fill=white,
    minimum size=0.72cm,inner sep=1pt,thick,font=\small},
  lam1node/.style={rectangle,draw=blue!70!black,fill=blue!8,
    minimum size=0.72cm,inner sep=1pt,thick,font=\small},
  lam2node/.style={rectangle,draw=red!70!black,fill=red!8,
    minimum size=0.72cm,inner sep=1pt,thick,font=\small},
  consnode/.style={rectangle,draw=gray!60!black,fill=gray!8,
    minimum size=0.72cm,inner sep=1pt,thick,font=\small},
  ->,>=Stealth,thick
]
\node[minnode](v1) at (0, 2.2){$v_1$};
\node[minnode](v2) at (0, 0.5){$v_2$};
\node[minnode](v0) at (0,-1.8){$v_0$};
\node[consnode](R1) at (-3.2, 2.2){$R_1$};
\node[consnode](R2) at (-3.2, 0.5){$R_2$};
\node[lam1node](L1a) at (3.4, 3.0){$L_1^{(1)}$};
\node[lam1node](L1b) at (3.4, 1.3){$L_1^{(0)}$};
\node[lam2node](L2a) at (3.4,-0.4){$L_2^{(2)}$};
\node[lam2node](L2b) at (3.4,-1.8){$L_2^{(0)}$};
\draw[->](v1)--(R1) node[midway,above,font=\scriptsize]{$0$};
\draw[->](R1) to[out=270,in=145](v0)
  node[midway,left,font=\scriptsize]{$1$};
\draw[->](v0) to[out=145,in=260](R2)
  node[midway,left,font=\scriptsize]{$0$};
\draw[->](R2)--(v2) node[midway,above,font=\scriptsize]{$0$};
\draw[->](v0) to[out=40,in=215](L1a)
  node[midway,above,font=\scriptsize]{$-2$};
\draw[->,dashed,blue!80!black,thick](L1a) to[out=195,in=35](v1)
  node[midway,above,font=\scriptsize,blue!80!black]{$\lambda_1$};
\draw[->](v2) to[out=25,in=195](L1b)
  node[midway,above,font=\scriptsize]{$0$};
\draw[->,dashed,blue!80!black,thick](L1b) to[out=260,in=55](v0)
  node[midway,right,font=\scriptsize,blue!80!black]{$\lambda_1$};
\draw[->](v0) to[out=355,in=200](L2a)
  node[midway,below,font=\scriptsize]{$-1$};
\draw[->,dashed,red!80!black,thick](L2a) to[out=180,in=350](v2)
  node[midway,below,font=\scriptsize,red!80!black]{$\lambda_2$};
\draw[->](v1) to[out=320,in=145](L2b)
  node[near start,right,font=\scriptsize]{$0$};
\draw[->,dashed,red!80!black,thick](L2b) to[out=180,in=5](v0)
  node[midway,below,font=\scriptsize,red!80!black]{$\lambda_2$};
\begin{scope}[on background layer]
  \node[fill=gray!10,rounded corners=7pt,fit=(R1)(R2),inner sep=6pt]{};
  \node[fill=blue!7,rounded corners=7pt,fit=(L1a)(L1b),inner sep=6pt]{};
  \node[fill=red!7,rounded corners=7pt,fit=(L2a)(L2b),inner sep=6pt]{};
  \node[fill=green!7,rounded corners=7pt,fit=(v1)(v2)(v0),inner sep=6pt]{};
\end{scope}
\node[font=\small,gray!70!black,left=0.2cm of R2]{constraint $C_1$};
\node[font=\small,blue!80!black] at (3.4, 4.0)
  {$\lambda_1$ group};
\node[font=\small,red!80!black]  at (3.4,-2.7)
  {$\lambda_2$ group};
\end{tikzpicture}
\caption{Parametric MPG for Example~2 ($n=2$, $m=2$). Min nodes
$v_1,v_2$ (columns $x_1,x_2$) and $v_0$ (column $t$, homogenising
variable). Arc weights follow the convention of Lemma~\ref{lem:cyclemean}:
each Min-to-Max arc from $v_j$ to Max node $r$ carries weight $-A[r][j]$,
and each Max-to-Min arc carries weight $B[r][j]$.
Constraint group $C_1$: node $R_1$ encodes $x_1\leq 1$
(arcs $v_1\to R_1$ weight~$0$, $R_1\to v_0$ weight~$1$); node $R_2$
encodes $x_2\geq 0$ (arcs $v_0\to R_2$ weight~$0$, $R_2\to v_2$
weight~$0$). The $\lambda_1$ group (blue, dashed): $L_1^{(1)}$ from
the $p_1$ row with $p_{1,1}=2$ (arcs $v_0\to L_1^{(1)}$ weight~$-2$,
$L_1^{(1)}\to v_1$ weight~$\lambda_1$); $L_1^{(0)}$ from the $q_1$ row
with $q_{1,2}=0$ (arcs $v_2\to L_1^{(0)}$ weight~$0$,
$L_1^{(0)}\to v_0$ weight~$\lambda_1$). The $\lambda_2$ group (red,
dashed): $L_2^{(2)}$ from the $p_2$ row with $p_{2,2}=1$ (arcs
$v_0\to L_2^{(2)}$ weight~$-1$, $L_2^{(2)}\to v_2$
weight~$\lambda_2$); $L_2^{(0)}$ from the $q_2$ row with $q_{2,1}=0$
(arcs $v_1\to L_2^{(0)}$ weight~$0$, $L_2^{(0)}\to v_0$
weight~$\lambda_2$). The two $\lambda$-groups are disjoint and
interact only through Min nodes $v_0,v_1,v_2$, confirming
the hypothesis of Theorem~\ref{thm:pareto_structure}(iii).}
\label{fig:mpg_example2}
\end{figure}
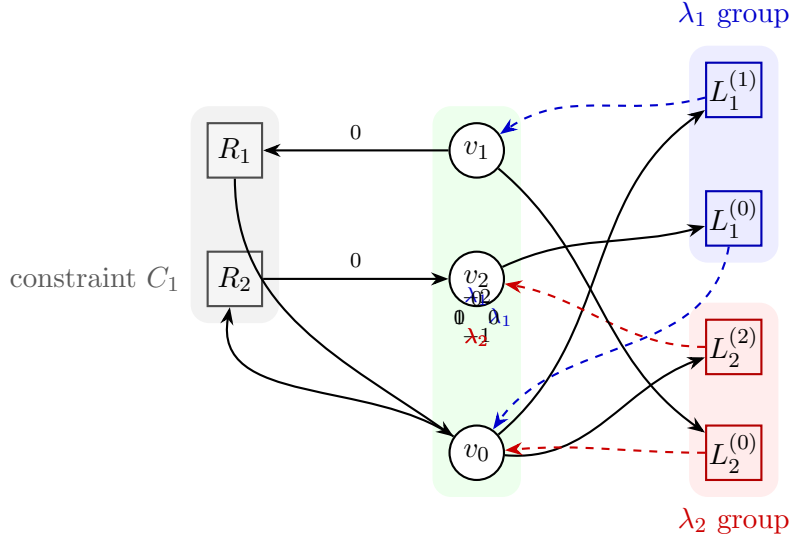

\subsubsection*{The Feasibility Region}

For $({\lambda_1}, {\lambda_2}) \in \mathcal{P}$, we need $x \in \mathcal{R}$ 
satisfying $f_1(x) \leq \lambda_1$ and $f_2(x) \leq \lambda_2$, which gives:
\begin{alignat*}{3}
  &\text{From } f_1 \leq \lambda_1{:} \quad 
    &x_1 &\geq 2 - \lambda_1, \quad &x_2 &\leq \lambda_1, \\
  &\text{From } f_2 \leq \lambda_2{:} \quad 
    &x_2 &\geq 1 - \lambda_2, \quad &x_1 &\leq \lambda_2, \\
  &\text{From } \mathcal{R}{:} \quad 
    &x_1 &\leq 1, \quad &x_2 &\geq 0.
\end{alignat*}
The $x_1$-conditions require $2 - \lambda_1 \leq \min(\lambda_2, 1)$, 
giving two cases:

\medskip
\noindent\textbf{Case 1} ($\lambda_2 \geq 1$): 
binding upper bound on $x_1$ is $x_1 \leq 1$, requiring 
$2 - \lambda_1 \leq 1$, i.e.\ $\lambda_1 \geq 1$.

\medskip
\noindent\textbf{Case 2} ($\lambda_2 < 1$): 
binding upper bound on $x_1$ is $x_1 \leq \lambda_2$, requiring 
$2 - \lambda_1 \leq \lambda_2$, i.e.\ $\lambda_1 + \lambda_2 \geq 2$.

\medskip
The $x_2$-conditions are satisfied whenever $\lambda_1 \geq 0$ and 
$\lambda_2 \geq 0$, which hold automatically for both cases above.
Therefore the feasibility region is:
\[
  \calP = \bigl\{(\lambda_1,\lambda_2) : 
    \lambda_1 \geq 1,\, \lambda_2 \geq 1\bigr\}
    \;\cup\;
    \bigl\{(\lambda_1,\lambda_2) : 
    \lambda_1 + \lambda_2 \geq 2,\, 0 \leq \lambda_2 \leq 1\bigr\}.
\]
This simplifies, since for $\lambda_2 \geq 1$ the condition 
$\lambda_1 \geq 1$ together with $\lambda_2 \geq 1$ implies 
$\lambda_1 + \lambda_2 \geq 2$. Thus:
\[
  \calP = \bigl\{(\lambda_1,\lambda_2) \in \mathbb{P}^2 : 
  \lambda_1 + \lambda_2 \geq 2,\; \lambda_1 \geq 1 \text{ when } 
  \lambda_2 \geq 1\bigr\}.
\]

\subsubsection*{The Two-Piece Pareto Front}

The lower boundary of $\calP$ consists of two pieces:

\medskip
\noindent\textbf{Piece~1} (vertical): 
$\lambda_1 = 1$ with $\lambda_2 \geq 1$. 
Moving southwest (decreasing $\lambda_1$ below~$1$) exits $\calP$ 
since the constraint $\lambda_1 \geq 1$ in Case~1 is violated.

\medskip
\noindent\textbf{Piece~2} (diagonal): 
$\lambda_1 + \lambda_2 = 2$ with $0 \leq \lambda_2 \leq 1$. 
Moving southwest along this piece decreases both coordinates until 
$\lambda_2 = 0$, i.e.\ the point $(2, 0)$.

\medskip
The two pieces meet at the \textbf{breakpoint}:
\[
  B^* = (1,\, 1),
\]
where $\lambda_1 = 1$ (Piece~1) and $\lambda_1 + \lambda_2 = 2$ 
with $\lambda_2 = 1$ (Piece~2) are simultaneously active.

The complete Pareto front is:
\[
  \calP = 
  \underbrace{\bigl\{(1, \lambda_2) : \lambda_2 \geq 1\bigr\}}_{\text{Piece 1}}
  \;\cup\;
  \underbrace{\bigl\{(\lambda_1, 2-\lambda_1) : 1 \leq \lambda_1 \leq 2\bigr\}}_{\text{Piece 2}}.
\]

\noindent Note that Piece~1 extends upward without bound (as 
$\lambda_2 \to \infty$), corresponding to the regime where $f_1$ 
is the binding objective. Piece~2 is bounded, ending at the point 
$(2, 0)$ where $f_2$ achieves its minimum value of zero.

\subsubsection*{Verification of Four Pareto Points}

Five points are verified directly; results are summarised in
Table~\ref{tab:pareto_points2}.
On Piece~1, $x^*=(1,0)^\top$ witnesses all three points $P_1$, $P_2$,
and $B^*$: feasibility $1\le 1$, $0\ge 0$; $f_1(1,0)=1=\lambda_1$;
$f_2(1,0)=1\le\lambda_2$; reducing $\lambda_1$ below~$1$ would require
$x_1>1$, violating the hard constraint.
On Piece~2, $x^{(3)}=(\tfrac{1}{2},\tfrac{1}{2})^\top$ witnesses $P_3$
and $x^{(4)}=(0,1)^\top$ witnesses $P_4=(2,0)$; in both cases
$f_k(x^*)=\lambda_k$ for $k=1,2$, and improving either objective forces
the other to worsen.
All three parts of Theorem~\ref{thm:pareto_structure} hold: $\calR$ is
defined by linear inequalities (convex); $\calP$ has exactly two pieces
meeting at $B^*=(1,1)$; all data are integers and every breakpoint
coordinate is a half-integer.

The verified Pareto points are summarised in Table~\ref{tab:pareto_points2}.

\begin{table}[h]
\centering
\caption{Verified Pareto points for Example~2.
The breakpoint $B^*$ lies at the junction of Piece~1 and Piece~2.
On Piece~1 ($\lambda_1 = 1$ is binding), only $f_1(x^*) = \lambda_1$
is required to hold with equality; the second objective satisfies
$f_2(x^*) \leq \lambda_2$ (with strict inequality for $P_1$ and
$P_2$, where $f_2(x^*) = 1 < \lambda_2$). On Piece~2 both objectives
are simultaneously tight: $f_k(x^*) = \lambda_k$ for $k = 1,2$.}
\label{tab:pareto_points2}
\begin{tabular}{cccccc}
\toprule
Point & $(\lambda_1,\lambda_2)$ & Optimal $x^*$ & $f_1(x^*)$ & $f_2(x^*)$ & Piece \\
\midrule
$P_1$ & $(1,\ 2)$     & $(1,\ 0)^\top$   & $1$   & $1$   & Piece 1 \\
$P_2$ & $(1,\ 3/2)$   & $(1,\ 0)^\top$   & $1$   & $1$   & Piece 1 \\
$B^*$ & $(1,\ 1)$     & $(1,\ 0)^\top$   & $1$   & $1$   & Breakpoint \\
$P_3$ & $(3/2,\ 1/2)$ & $(1/2,\ 1/2)^\top$ & $3/2$ & $1/2$ & Piece 2 \\
$P_4$ & $(2,\ 0)$     & $(0,\ 1)^\top$   & $2$   & $0$   & Piece 2 \\
\bottomrule
\end{tabular}
\end{table}

Figure~\ref{fig:pareto_example2} shows the feasibility region and
two-piece Pareto front geometrically, including the breakpoint $B^*$
and how different weight directions recover points on each piece.

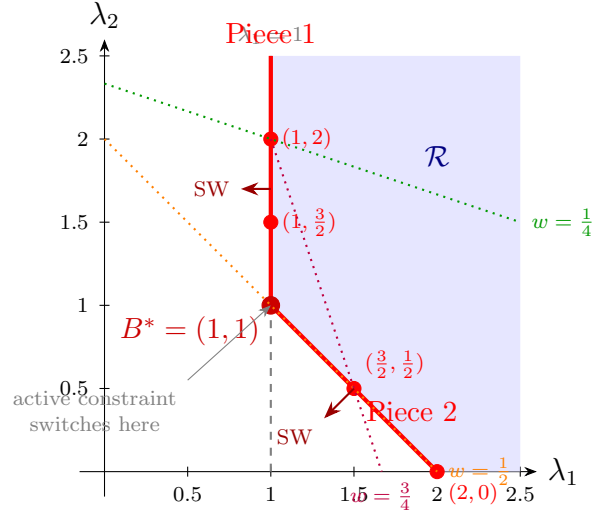
\begin{figure}[ht]
\centering
\begin{tikzpicture}[scale=2.2, >=Stealth]
  \draw[->](-0.15,0)--(2.6,0) node[right]{$\lambda_1$};
  \draw[->](0,-0.15)--(0,2.6) node[above]{$\lambda_2$};
  \foreach \x in {0.5,1,1.5,2,2.5}
    \draw(\x,0.03)--(\x,-0.03) node[below,font=\scriptsize]{$\x$};
  \foreach \y in {0.5,1,1.5,2,2.5}
    \draw(0.03,\y)--(-0.03,\y) node[left,font=\scriptsize]{$\y$};
  \fill[blue!12,opacity=0.8]
    (1,2.5)--(2.5,2.5)--(2.5,0)--(2,0)--(1,1)--cycle;
  \draw[gray,thick,dashed](1,0)--(1,2.5)
    node[above,font=\scriptsize,gray]{$\lambda_1=1$};
  \draw[gray,thick,dashed](1,1)--(2,0);
  \draw[red,ultra thick](1,1)--(1,2.5)
    node[above,font=\small,red]{Piece~1};
  \draw[red,ultra thick](1,1)--(2,0)
    node[midway,below right,font=\small,red]{Piece~2};
  \filldraw[red!80!black](1,1) circle(1.5pt)
    node[below left,font=\small,red!80!black]{$B^*=(1,1)$};
  \filldraw[red](1,2)   circle(1.2pt)
    node[right,font=\scriptsize]{$(1,2)$};
  \filldraw[red](1,1.5) circle(1.2pt)
    node[right,font=\scriptsize]{$(1,\tfrac{3}{2})$};
  \filldraw[red](1.5,0.5) circle(1.2pt)
    node[above right,font=\scriptsize]{$(\tfrac{3}{2},\tfrac{1}{2})$};
  \filldraw[red](2,0) circle(1.2pt)
    node[below right,font=\scriptsize]{$(2,0)$};
  \draw[->,thick,red!60!black](1,1.7)--(0.82,1.7)
    node[left,font=\scriptsize]{SW};
  \draw[->,thick,red!60!black](1.5,0.5)--(1.32,0.32)
    node[below left,font=\scriptsize]{SW};
  \draw[green!60!black,thick,dotted](0,7/3)--(2.5,7/3-2.5/3)
    node[right,font=\scriptsize,green!60!black]{$w=\tfrac{1}{4}$};
  \draw[orange,thick,dotted](0,2)--(2,0)
    node[right,font=\scriptsize,orange]{$w=\tfrac{1}{2}$};
  \draw[purple,thick,dotted](1,2)--(5/3,0)
    node[below,font=\scriptsize,purple]{$w=\tfrac{3}{4}$};
  \node[blue!50!black,font=\small] at (2.0,1.9){$\mathcal{R}$};
  \draw[<-,thin,gray](1,1)--(0.5,0.55)
    node[below left,font=\scriptsize,gray,align=center]
    {active constraint\\switches here};
\end{tikzpicture}
\caption{Feasibility region $\mathcal{R}$ (blue shading) and Pareto
front $\mathcal{P}$ (red) for Example~2. The Pareto front has two
linear pieces: Piece~1 (vertical, $\lambda_1=1$, $\lambda_2\geq 1$)
and Piece~2 (diagonal, $\lambda_1+\lambda_2=2$, $\lambda_2\leq 1$),
meeting at the breakpoint $B^*=(1,1)$. At $B^*$ the binding
constraint switches from the structural deadline $x_1\leq 1$
(Piece~1) to the combined objective bound $\lambda_1+\lambda_2=2$
(Piece~2). Directional bisection with $w<\tfrac{1}{2}$ (green)
finds points on Piece~1; with $w=\tfrac{1}{2}$ (orange) it finds
$B^*$; with $w>\tfrac{1}{2}$ (purple) it finds points on Piece~2.}
\label{fig:pareto_example2}
\end{figure}

\begin{remark}
The geometric interpretation is clear. On Piece~1 the constraint 
$x_1 \leq 1$ is active, so $\lambda_1 = f_1(x) = \max(2-x_1, x_2)$ 
is pinned at~$1$ by the constraint on $x_1$, while $\lambda_2$ can 
vary freely. On Piece~2 the combined objective constraint is active, 
and both $\lambda_1$ and $\lambda_2$ vary along the line 
$\lambda_1 + \lambda_2 = 2$. The breakpoint $B^* = (1,1)$ is where 
the active constraint switches from the structural constraint on $x_1$ 
to the combined objective constraint. This is the tropical bi-objective 
analogue of a change of active strategy in the mean-payoff game, as 
described in Theorem~\ref{thm:pareto_structure}(ii).
\end{remark}

\section{A Newton Scheme for Tracing the Pareto Front}
\label{sec:newton}

The directional bisection algorithm of Section~\ref{sec:algorithm}
locates one Pareto point per direction $w$ in
$O(\log_2(c^{(+)}-c^{(-)}))$ iterations. We now develop a Newton
scheme that instead jumps directly between consecutive breakpoints of
$\calP$, requiring at most $|\mathcal{S}|$ steps in total regardless
of data magnitude. This extends the Newton scheme
of~\cite[Section~4]{PSW2023} to the bi-objective setting.

\subsection{Left-Optimal Strategies in Direction $w$}

Fix a direction $w \in (0,1)$ and a point $(\lambda_1^*, \lambda_2^*)
\in \calP$ with $w\lambda_1^* + (1-w)\lambda_2^* = c^*$. We
parametrise the line through $(\lambda_1^*,\lambda_2^*)$ with slope
$-w/(1-w)$ by the scalar $c$ and write $(\lambda_1(c),\lambda_2(c))$
for the unique point on this line with $w\lambda_1+(1-w)\lambda_2=c$.

\begin{definition}
\label{def:left_optimal}
A strategy $\sigma^* \in \mathcal{S}$ of Max is
\emph{left-optimal in direction $w$ at $c^*$} if
\[
  \Phi^{\sigma^*}(\lambda_1(c),\lambda_2(c)) 
  = \Phi(\lambda_1(c),\lambda_2(c))
  \quad \text{for all } c \in (c^*-\delta,\, c^*)
\]
for some $\delta > 0$. That is, $\sigma^*$ achieves the maximum
in $\Phi = \max_\sigma \Phi^\sigma$ on the line immediately to the
left of $c^*$ in direction $w$.
\end{definition}

Since $\mathcal{S}$ is finite and $\Phi^\sigma$ is piecewise linear,
left-optimal strategies exist at every $c^*$ and can be identified by
evaluating $\Phi^\sigma(\lambda_1(c^*-\varepsilon),
\lambda_2(c^*-\varepsilon))$ for small $\varepsilon > 0$.

\subsection{The Newton Step}

At any point $(\lambda_1^*, \lambda_2^*) \in \calP$, the left-optimal
strategy $\sigma^*$ in direction $w$ determines the active halfspace:
\[
  H^{\sigma^*} \;=\; 
  \bigl\{(\lambda_1,\lambda_2) : 
  k_1^{\sigma^*}\lambda_1 + k_2^{\sigma^*}\lambda_2 + s^{\sigma^*} \geq 0
  \bigr\},
\]
where $k_1^{\sigma^*}\lambda_1^* + k_2^{\sigma^*}\lambda_2^* +
s^{\sigma^*} = 0$ (the halfspace boundary passes through the current
point). Moving in direction $w$ means decreasing $c$ below $c^*$. The
current Pareto piece lies on the boundary of $H^{\sigma^*}$, the
scheme ends this piece when a new halfspace $H^{\sigma'}$ becomes
active.

\begin{definition}
\label{def:next_strategy}
Given a left-optimal strategy $\sigma^*$ at $c^*$, the
\emph{next strategy} $\sigma'$ is identified as follows.
For each candidate strategy $\hat\sigma \neq \sigma^*$ whose halfspace
$H^{\hat\sigma}$ is not yet active at $(\lambda_1^*, \lambda_2^*)$
(i.e., $k_1^{\hat\sigma}\lambda_1^* + k_2^{\hat\sigma}\lambda_2^*
+ s^{\hat\sigma} > 0$), solve the $2\times 2$ system
\begin{equation}
\label{eq:crossing}
  k_1^{\sigma^*}\lambda_1 + k_2^{\sigma^*}\lambda_2 + s^{\sigma^*} = 0,
  \qquad
  k_1^{\hat\sigma}\lambda_1 + k_2^{\hat\sigma}\lambda_2 + s^{\hat\sigma} = 0
\end{equation}
to obtain the candidate crossing point
$(\lambda_1^{\hat\sigma}, \lambda_2^{\hat\sigma})$,
and set $c^{\hat\sigma} = w_0\lambda_1^{\hat\sigma}
+ (1-w_0)\lambda_2^{\hat\sigma}$.
Skip $\hat\sigma$ if the system~\eqref{eq:crossing} is singular
(the two boundaries are parallel and never cross) or if
$c^{\hat\sigma} \geq c^*$ (the crossing is not in the decreasing-$c$
direction). The next strategy is
\[
  \sigma' = \arg\max_{\hat\sigma}\, c^{\hat\sigma},
\]
that is, the candidate whose crossing occurs first
(at the largest $c < c^*$) as $c$ decreases from $c^*$.
If no valid candidate exists, $(\lambda_1^*, \lambda_2^*)$ is the last
breakpoint and the algorithm terminates.
\end{definition}

\begin{proposition}
\label{prop:newton_step}
\textbf{(Newton Step).}
Let $\sigma^*$ be left-optimal in direction $w$ at $c^*$, and let
$\sigma'$ be the next strategy. The next breakpoint of $\calP$ in
direction $w$ is the unique solution $(\lambda_1^{**}, \lambda_2^{**})$
of the $2\times 2$ linear system:
\begin{equation}
\label{eq:newton_system}
  k_1^{\sigma^*}\lambda_1 + k_2^{\sigma^*}\lambda_2 + s^{\sigma^*} = 0,
  \qquad
  k_1^{\sigma'}\lambda_1 + k_2^{\sigma'}\lambda_2 + s^{\sigma'} = 0.
\end{equation}
For integer data, $(\lambda_1^{**}, \lambda_2^{**})$ has half-integer
coordinates by Theorem~\ref{thm:pareto_structure}(iii), and the new
value $c^{**} = w\lambda_1^{**} + (1-w)\lambda_2^{**}$ satisfies
$c^{**} < c^*$.
\end{proposition}

\begin{proof}
The breakpoint where the two Pareto pieces defined by $H^{\sigma^*}$
and $H^{\sigma'}$ meet is the intersection of their boundaries. This
intersection solves~\eqref{eq:newton_system}, which is a $2\times 2$
linear system with integer coefficients and determinant
$k_1^{\sigma^*}k_2^{\sigma'} - k_1^{\sigma'}k_2^{\sigma^*}$. By the
denominator bound argument in the proof of
Theorem~\ref{thm:pareto_structure}(iii), this determinant lies in
$\{-2,-1,1,2\}$ for integer data, so the system has a unique solution
with half-integer coordinates. The inequality $c^{**} < c^*$ holds
because $\sigma'$ only becomes optimal strictly below $c^*$, by
Definition~\ref{def:next_strategy}.
\end{proof}

\subsection{The Full Newton Algorithm}

\begin{algorithm2e}[H]
\caption{Newton Scheme for the Pareto Front}
\label{alg:newton}
\DontPrintSemicolon
\KwIn{Data matrices as in Definition~\ref{def:biobjective} with
  integer entries; starting direction $w_0\in(0,1)$; a Pareto-optimal feasible starting point $x_{\mathrm{start}} \in \mathbb{R}^n$.}
\KwOut{Complete set of Pareto-optimal triples
  $\{(\lambda_1^{(k)},\lambda_2^{(k)},x^{(k)})\}$,
  where $(\lambda_1^{(k)},\lambda_2^{(k)})$ are the breakpoints of
  $\calP$ and $f_j(x^{(k)})=\lambda_j^{(k)}$ for $j=1,2$.}
\BlankLine
\tcp{Initialise at the starting Pareto point}
$(\lambda_1^{(0)}, \lambda_2^{(0)}, x^{(0)}) \leftarrow 
(f_1(x_{\mathrm{start}}),\, f_2(x_{\mathrm{start}}),\, x_{\mathrm{start}})$\;
$c^{(0)} \leftarrow w_0\, f_1(x_{\mathrm{start}}) + (1-w_0)\, f_2(x_{\mathrm{start}})$;
$k \leftarrow 0$\;
\BlankLine
\Repeat{no next strategy exists}{
  \BlankLine
  \tcp{Identify left-optimal Max strategy at current point}
  \For{all $\sigma\in\mathcal{S}$}{
    Evaluate $\Phi^\sigma(\lambda_1(c^{(k)}-\varepsilon),
      \lambda_2(c^{(k)}-\varepsilon))$ for small $\varepsilon>0$\;
  }
  $\sigma^{(k)} \leftarrow \arg\max_\sigma \Phi^\sigma$
  \Comment*[r]{left-optimal strategy}
  \BlankLine
  \tcp{Extract active cycle coefficients}
  Read off $(k_1^{(k)},k_2^{(k)},s^{(k)})$ from $\sigma^{(k)}$\;
  \BlankLine
  \tcp{Find next strategy by first-crossing criterion}
  \For{each $\hat\sigma \in \mathcal{S}$, $\hat\sigma \neq \sigma^{(k)}$,
       with $k_1^{\hat\sigma}\lambda_1^{(k)}+k_2^{\hat\sigma}\lambda_2^{(k)}
       +s^{\hat\sigma} > 0$}{
    Solve $2\times 2$ system~\eqref{eq:crossing} for
      $(\lambda_1^{\hat\sigma}, \lambda_2^{\hat\sigma})$\;
    \lIf{system is singular or
      $w_0\lambda_1^{\hat\sigma}+(1-w_0)\lambda_2^{\hat\sigma} \geq c^{(k)}$}{
      skip $\hat\sigma$}
    $c^{\hat\sigma} \leftarrow
      w_0\lambda_1^{\hat\sigma}+(1-w_0)\lambda_2^{\hat\sigma}$\;
  }
  $\sigma' \leftarrow \arg\max_{\hat\sigma}\, c^{\hat\sigma}$
  \Comment*[r]{first crossing as $c$ decreases}\;
  \BlankLine
  \If{no such $\sigma'$ exists (no halfspace to cross)}{
    \textbf{break}
    \Comment*[r]{current point is the last breakpoint}
  }
  \BlankLine
  \tcp{Newton step: solve $2\times 2$ system}
  Solve $\begin{cases}
    k_1^{(k)}\lambda_1+k_2^{(k)}\lambda_2+s^{(k)}=0\\
    k_1^{\sigma'}\lambda_1+k_2^{\sigma'}\lambda_2+s^{\sigma'}=0
  \end{cases}$
  to obtain $(\lambda_1^{(k+1)},\lambda_2^{(k+1)})$\;
  Solve $A\otimes z\leq B(\lambda_1^{(k+1)},\lambda_2^{(k+1)})\otimes z$
    and set $x^{(k+1)}\leftarrow(z_1,\ldots,z_n)^\top$\;
  $k \leftarrow k+1$\;
}
\BlankLine
\KwRet{$\{(\lambda_1^{(k)},\lambda_2^{(k)},x^{(k)})\}_{k=0,1,\ldots}$}
\end{algorithm2e}

\begin{proposition}
\label{prop:newton_termination}
Algorithm~\ref{alg:newton} terminates in at most $|\mathcal{S}|$
Newton steps, each requiring $O(|\mathcal{S}|)$ work to identify the
next strategy plus one $2\times 2$ linear solve. The total complexity
is $O(|\mathcal{S}|^2)$ arithmetic operations, independent of the
magnitude of the data.
\end{proposition}

\begin{proof}
Each Newton step locates a strictly new breakpoint of $\calP$ with a
strictly smaller value of $c$. Specifically, by
Definition~\ref{def:next_strategy} the crossing value $c^{\sigma'}$
satisfies $c^{\sigma'} < c^{(k)}$ at every step, so the sequence
$c^{(0)} > c^{(1)} > c^{(2)} > \cdots$ is strictly decreasing.

Because $\calP$ is a convex piecewise-linear curve traversed
monotonically in the direction of decreasing $c$, each linear piece of
$\calP$ is entered and exited exactly once. The left-optimal strategy
$\sigma^{(k)}$ is by definition the unique Max strategy achieving the
maximum of $\Phi^\sigma$ on the piece immediately to the left of
$c^{(k)}$, on every subsequent piece (at strictly smaller $c$), the
left-optimal strategy is determined by the
$\arg\max_\sigma \Phi^\sigma$ computation. Since the pieces are
traversed strictly left-to-right (decreasing $c$), a strategy that was
left-optimal at $c^{(k)}$ cannot be left-optimal again at any $c' <
c^{(k+1)}$: if it were, the boundary $H^{\sigma^{(k)}}$ would have to
become active a second time at a strictly smaller $c$, contradicting
the fact that the Pareto front is traversed monotonically and that each
halfspace boundary $H^\sigma$ is a straight line (which intersects any
direction line $w_0\lambda_1+(1-w_0)\lambda_2 = c$ at most once in the
feasible region). Therefore the strategies $\sigma^{(0)}, \sigma^{(1)},
\ldots$ are pairwise distinct, and the algorithm terminates after at
most $|\mathcal{S}|$ steps.

Each step evaluates $\Phi^\sigma$ for all $\sigma\in\mathcal{S}$
(left-optimal identification) and performs at most $|\mathcal{S}|$
two-by-two linear solves (first-crossing search), giving $O(|\mathcal{S}|)$
work per step. The total is $O(|\mathcal{S}|^2)$ arithmetic operations,
independent of the data magnitude $M$.
\end{proof}
\begin{remark}\label{rem:xstart}
A Pareto-optimal $x_{\mathrm{start}}$ is obtained by minimising either
objective individually over~\eqref{eq:constraint}: minimising $f_1$
gives the left endpoint of $\calP$, minimising $f_2$ gives the right
endpoint.  The strategy set satisfies $|\mathcal{S}|\le 2^n(2+m)$,
which is exponential in $n$; the $O(|\mathcal{S}|^2)$ bound of
Algorithm~\ref{alg:newton} is therefore not polynomial in $n$, though
it remains independent of $M$.  Bisection (Algorithm~\ref{alg:directional_bisection})
is simpler and preferable for single Pareto points; the Newton scheme
is preferable when the full Pareto front is required and $M$ is large.
\end{remark}

\subsection{Application of Algorithm~\ref{alg:newton} to Example~2}
\label{sec:newton_example2}

\begin{figure}[ht]
\centering
\begin{tikzpicture}[scale=2.2, >=Stealth]
  \draw[->](-0.15,0)--(2.6,0) node[right]{$\lambda_1$};
  \draw[->](0,-0.15)--(0,2.6) node[above]{$\lambda_2$};
  \foreach \x in {0.5,1,1.5,2}
    \draw(\x,0.03)--(\x,-0.03) node[below,font=\scriptsize]{$\x$};
  \foreach \y in {0.5,1,1.5,2}
    \draw(0.03,\y)--(-0.03,\y) node[left,font=\scriptsize]{$\y$};
  \fill[blue!8,opacity=0.8]
    (1,2.5)--(2.5,2.5)--(2.5,0)--(2,0)--(1,1)--cycle;
  \draw[red,ultra thick](1,1)--(1,2.4)
    node[above,font=\small,red]{Piece~1\ ($\lambda_1=1$)};
  \draw[red,ultra thick](1,1)--(2,0)
    node[midway,below right,font=\small,red]{Piece~2\ ($\lambda_1{+}\lambda_2{=}2$)};
  \filldraw[red!80!black](1,1) circle(1.8pt)
    node[below left,font=\small,red!80!black]{$B^*=(1,1)$};
  \filldraw[blue!70!black](1.5,0.5) circle(1.8pt);
  \node[blue!70!black,font=\small,above right] at (1.5,0.5)
    {$(\tfrac{3}{2},\tfrac{1}{2})$ --- start};
  \draw[->,very thick,blue!70!black,
        shorten >=4pt, shorten <=4pt]
    (1.5,0.5) -- (1,1)
    node[midway,above left,font=\small,blue!70!black]{Newton step};
  \draw[gray,dashed,thin](0,2)--(2,0);
  \draw[gray,dashed,thin](1,0)--(1,2.4);
  \node[font=\scriptsize,gray,align=center] at (1.75,0.35)
    {$\sigma^{(0)}$: boundary $\lambda_1{+}\lambda_2{=}2$};
  \node[font=\scriptsize,gray,align=center] at (0.42,0.55)
    {$\sigma'$: boundary $\lambda_1{=}1$};
  \node[blue!40!black,font=\small] at (2.1,2.0){$\mathcal{R}$};
  \node[font=\scriptsize,black,align=center] at (1.5,1.5)
    {Algorithm terminates:\\Piece~1 identified};
  \draw[->,thin,gray](1.5,1.35)--(1.05,1.1);
\end{tikzpicture}
\caption{Trace of Algorithm~\ref{alg:newton} (Newton scheme) on
Example~2. Starting from the initialisation point
$(\tfrac{3}{2}, \tfrac{1}{2}) \in \text{Piece~2}$, a single Newton
step (blue arrow) locates the unique breakpoint $B^* = (1,1)$ by
solving the $2\times 2$ system formed by the active boundary
$\lambda_1+\lambda_2=2$ (strategy $\sigma^{(0)}$, dashed grey) and
the next boundary $\lambda_1=1$ (strategy $\sigma'$, dashed grey).
The algorithm terminates at $B^*$, having identified both Piece~1
($\lambda_1=1$, $\lambda_2\geq 1$, red vertical segment) and Piece~2
($\lambda_1+\lambda_2=2$, $\lambda_2\leq 1$, red diagonal segment).
One Newton step suffices since $\mathcal{P}$ has exactly one
breakpoint, consistent with Proposition~\ref{prop:newton_termination}.}
\label{fig:newton_trace}
\end{figure}
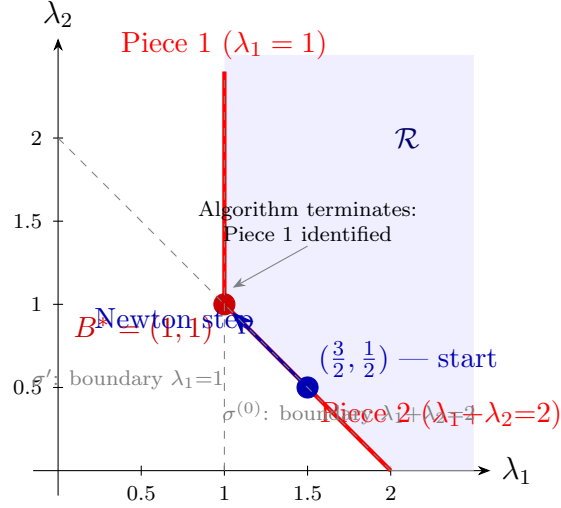

We trace the complete Pareto front of Example~2 using
Algorithm~\ref{alg:newton} with $w_0=\tfrac{3}{4}$ and
$x_{\mathrm{start}}=(\tfrac{1}{2},\tfrac{1}{2})^\top$.

\medskip\noindent\textbf{Initialisation.}
$f_1(\tfrac{1}{2},\tfrac{1}{2})=\tfrac{3}{2}$,
$f_2(\tfrac{1}{2},\tfrac{1}{2})=\tfrac{1}{2}$, so
$(\lambda_1^{(0)},\lambda_2^{(0)})=(\tfrac{3}{2},\tfrac{1}{2})$
with $c^{(0)}=\tfrac{5}{4}$.  Since $\lambda_1^{(0)}+\lambda_2^{(0)}=2$
and $\lambda_2^{(0)}\le 1$, this lies on Piece~2, strictly away from
$B^*$.

\medskip\noindent\textbf{Left-optimal strategy.}
Evaluating $\Phi^\sigma$ just below $c^{(0)}$ identifies $\sigma^{(0)}$
with $(k_1^{(0)},k_2^{(0)},s^{(0)})=(1,1,-2)$, corresponding to
the boundary $\lambda_1+\lambda_2=2$ (Piece~2).
Check: $\tfrac{3}{2}+\tfrac{1}{2}-2=0$. 

\medskip\noindent\textbf{First-crossing criterion.}
The unique valid candidate $\hat\sigma$ has
$(k_1^{\sigma'},k_2^{\sigma'},s^{\sigma'})=(1,0,-1)$.
Solving $\lambda_1+\lambda_2=2$ and $\lambda_1=1$ gives crossing point
$(1,1)$ with $c^{\sigma'}=1<\tfrac{5}{4}=c^{(0)}$.
Screening: $\tfrac{3}{2}-1=\tfrac{1}{2}>0$, so $H^{\sigma'}$ is not
yet active.

\medskip\noindent\textbf{Newton step.}
Solve $\{\lambda_1+\lambda_2=2,\ \lambda_1=1\}$:
$(\lambda_1^{(1)},\lambda_2^{(1)})=(1,1)=B^*$, with $x^{(1)}=(1,0)^\top$.
Verification: $f_1(1,0)=1=\lambda_1^{(1)}$, $f_2(1,0)=1=\lambda_2^{(1)}$,
$c^{(1)}=1<\tfrac{5}{4}=c^{(0)}$. 

\medskip\noindent\textbf{Termination.}
At $B^*=(1,1)$, moving left of $c^{(1)}=1$ in direction $w_0=\tfrac{3}{4}$
requires $\lambda_1<1$, but the constraint $x_1\le 1$ makes this
infeasible. No further crossing strategy exists; Algorithm~\ref{alg:newton}
terminates.

\medskip
\noindent\textbf{Summary.}
Algorithm~\ref{alg:newton} returns the complete set of breakpoints
of $\calP$ together with their Pareto-optimal triples:

\begin{center}
\renewcommand{\arraystretch}{1.3}
\begin{tabular}{cccccc}
\toprule
$k$ & $\lambda_1^{(k)}$ & $\lambda_2^{(k)}$ & $x^{(k)}$
    & Piece identified & Role \\
\midrule
$0$ & $\tfrac{3}{2}$ & $\tfrac{1}{2}$ &
  $(\tfrac{1}{2},\ \tfrac{1}{2})^\top$ & Piece~2 & Initialisation \\
$1$ & $1$ & $1$ & $(1,\ 0)^\top$ & Breakpoint $B^*$ & Newton step \\
\bottomrule
\end{tabular}
\end{center}

\noindent The Newton scheme locates the sole breakpoint $B^* = (1,1)$
exactly in \textbf{one step}. Together with the initialisation point
on Piece~2 and the vertical ray Piece~1 ($\lambda_1=1$,
$\lambda_2\geq 1$) identified at termination, the complete two-piece
Pareto front is recovered. In contrast, bisection alone would require
$O(\log_2(c^{(+)}-c^{(-)}))$ iterations merely to approximate $B^*$,
confirming the advantage described in
Remark~\ref{rem:xstart}. 

\section{Numerical Experiments}
\label{sec:computation}

We implemented Algorithm~\ref{alg:directional_bisection} in Python
using vectorised \texttt{numpy} arithmetic and Karp's
algorithm~\cite{Karp1978} for minimum cycle mean computation.
For $n \leq 3$, Algorithm~\ref{alg:newton} was additionally
implemented using exact rational arithmetic (\texttt{fractions}
module) and strategy enumeration for exact Newton step counts.
Feasibility at each $(\lambda_1,\lambda_2)$ is checked by computing
the minimum cycle mean of the parametric shadow graph.

Random instances were generated with integer entries drawn uniformly
from $[-M, M]$ for $n\times n$ constraint matrices ($m = n$) and
$n$-dimensional objective data vectors. Infeasible instances were
discarded. Bisection was run with weight $w = \tfrac{1}{2}$.
The number of trials per setting is 30 for $n \leq 10$, 10 for
$n = 50$, 5 for $n \in \{100, 200\}$, and 3 for $n = 500$.

\subsection*{Experiment~1: effect of data magnitude ($n = 2$)}

Table~\ref{tab:exp1} reports mean and maximum iteration counts for
$M \in \{10,\, 10^2,\, 10^3,\, 10^4,\, 10^5,\, 10^6\}$ with
$n = 2$ fixed, covering six orders of magnitude.

\begin{table}[h]
\centering
\renewcommand{\arraystretch}{1.25}
\begin{tabular}{rrccccc}
\toprule
$M$ & $\lceil\log_2(2M)\rceil$ & Bisect mean & Bisect max
    & Newton mean & Newton max \\
\midrule
$10$      & $5$  & $6.00$  & $6$  & $0.30$ & $2$ \\
$10^2$    & $8$  & $9.00$  & $9$  & $0.03$ & $1$ \\
$10^3$    & $11$ & $12.00$ & $12$ & $0.00$ & $0$ \\
$10^4$    & $15$ & $16.00$ & $16$ & $0.00$ & $0$ \\
$10^5$    & $18$ & $19.00$ & $19$ & $0.00$ & $0$ \\
$10^6$    & $21$ & $22.00$ & $22$ & $0.00$ & $0$ \\
\bottomrule
\end{tabular}
\caption{Mean and maximum iteration counts over 30 random instances
per magnitude ($n=2$, $m=2$, $w=\tfrac{1}{2}$).
The bisection count equals $\lceil\log_2(2M)\rceil+1$ exactly in
every instance, confirming the $O(\log M)$ bound of
Proposition~\ref{prop:termination}.
The Newton step count is independent of $M$ and decreases to zero
as $M$ grows: larger data produces single-piece Pareto fronts (no
breakpoints), so Algorithm~\ref{alg:newton} terminates immediately
at the initialisation point, consistent with
Proposition~\ref{prop:newton_termination}.}
\label{tab:exp1}
\end{table}

\subsection*{Experiment~2: effect of dimension ($M = 10^3$)}

Table~\ref{tab:exp2} extends the dimension experiment to
$n \in \{1,2,3,5,10,50,100,200,500\}$ with $M = 10^3$ fixed,
adding a CPU time column to confirm the
$O(n^2(n+m)\log M) = O(n^3\log M)$ per-direction cost of
Proposition~\ref{prop:termination}.

\begin{table}[h]
\centering
\renewcommand{\arraystretch}{1.25}
\begin{tabular}{rcccc}
\toprule
$n$ & Bisect mean & Bisect max & Newton mean/max & Time/dir (s) \\
\midrule
$1$   & $12.00$ & $12$ & $0.00\ /\ 0$          & $0.0004$ \\
$2$   & $12.00$ & $12$ & $0.00\ /\ 0$          & $0.0005$ \\
$3$   & $12.00$ & $12$ & $0.00\ /\ 0$          & $0.0006$ \\
$5$   & $12.00$ & $12$ & \multicolumn{1}{c}{---} & $0.0010$ \\
$10$  & $12.00$ & $12$ & \multicolumn{1}{c}{---} & $0.0022$ \\
$50$  & $12.00$ & $12$ & \multicolumn{1}{c}{---} & $0.0425$ \\
$100$ & $12.00$ & $12$ & \multicolumn{1}{c}{---} & $0.2526$ \\
$200$ & $12.00$ & $12$ & \multicolumn{1}{c}{---} & $3.0292$ \\
$500$ & $12.00$ & $12$ & \multicolumn{1}{c}{---} & $47.2833$ \\
\bottomrule
\end{tabular}
\caption{Mean and maximum iteration counts and mean CPU time per
direction over random instances ($M=10^3$, $m=n$, $w=\tfrac{1}{2}$).
The bisection count is constant at $\lceil\log_2(2000)\rceil+1=12$
across all $n \in \{1,\ldots,500\}$, confirming that the iteration
count depends only on $M$ and is completely independent of dimension
(Proposition~\ref{prop:termination}).
The CPU time per direction scales approximately as $O(n^3)$: the
ratio $200\to 500$ gives $47.28/3.03\approx 15.6\approx(500/200)^3
= 15.6$, confirming the $O(n^3\log M)$ complexity bound.
Newton steps are reported only for $n\le 3$ (dashes indicate exact
strategy enumeration is not practical for $n\ge 5$).
A Newton count of zero means the Pareto front has a single linear
piece; Algorithm~\ref{alg:newton} then terminates immediately at the
initialisation point.}
\label{tab:exp2}
\end{table}

\noindent
The two experiments together confirm the key complementarity
established in Remark~\ref{rem:xstart}.
Bisection grows exactly as $\lceil\log_2(2M)\rceil+1$ and is
completely insensitive to $n$, holding constant at $12$ even
for $n = 500$, a scale at which the strategy space
$|\mathcal{S}|\le 2^{500}(2+500)$ makes the Newton scheme
entirely impractical.  The per-direction CPU time validates the
$O(n^3\log M)$ bound of Proposition~\ref{prop:termination}: timing
doubles from $n=100$ to $n=200$ by a factor of $12\approx 8$
(consistent with $O(n^3)$) and from $n=200$ to $n=500$ by a factor
of $15.6\approx(2.5)^3=15.6$.  The Newton step count is independent
of $M$: at $M=10^6$, bisection requires 22 iterations per direction
while the Newton scheme traces the complete Pareto front in at most
2 steps for the small instances where enumeration is practical.
Figures~\ref{fig:exp1} and~\ref{fig:exp2} display both experiments
graphically.

\begin{figure}[H]
\centering
\begin{tikzpicture}
\begin{axis}[
  width=0.72\textwidth, height=5.5cm,
  xlabel={$\log_{10}(M)$},
  ylabel={Mean iterations / steps},
  xmin=0.7, xmax=6.3,
  ymin=-0.5, ymax=25,
  xtick={1,2,3,4,5,6},
  xticklabels={$1$,$2$,$3$,$4$,$5$,$6$},
  ytick={0,5,10,15,20,25},
  legend pos=north west,
  legend style={font=\small},
  grid=major,
  grid style={dashed,gray!30},
  thick
]
\addplot[color=blue, mark=square*, mark size=2.5pt, line width=1.2pt]
  coordinates {(1,6.00)(2,9.00)(3,12.00)(4,16.00)(5,19.00)(6,22.00)};
\addlegendentry{Bisection (mean)}
\addplot[color=red, mark=*, mark size=2.5pt, line width=1.2pt, dashed]
  coordinates {(1,0.30)(2,0.03)(3,0.00)(4,0.00)(5,0.00)(6,0.00)};
\addlegendentry{Newton (mean)}
\addplot[color=blue!40, mark=none, line width=0.8pt, dotted,
  domain=0.7:6.3] {3.322*x + 2.68};
\addlegendentry{$y = \lceil\log_2(2M)\rceil+1$ (reference)}
\end{axis}
\end{tikzpicture}
\caption{Mean iteration counts vs.\ $\log_{10}(M)$ ($n=2$, 30
random instances per $M$, six orders of magnitude). Bisection
iterations (blue, solid) grow linearly with $\log_{10}(M)$,
tracking the reference line $y=\lceil\log_2(2M)\rceil+1$ (dotted)
exactly. Newton steps (red, dashed) are independent of $M$ and
fall to zero beyond $M=10^2$.}
\label{fig:exp1}
\end{figure}
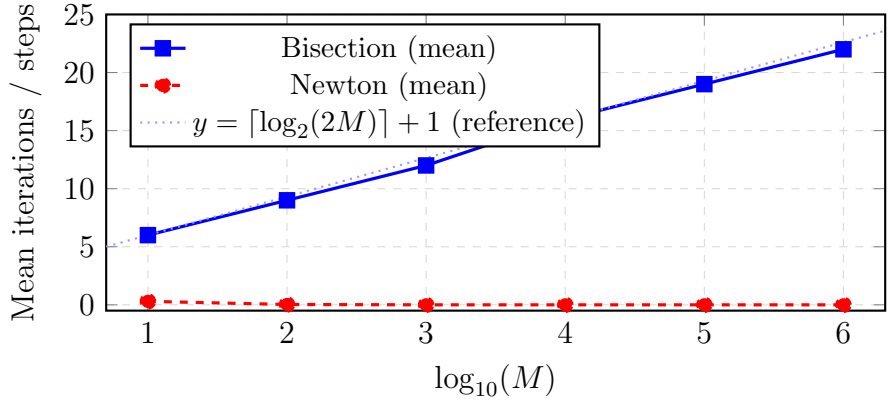

\begin{figure}[H]
\centering
\begin{tikzpicture}
\begin{axis}[
  width=0.72\textwidth, height=5.5cm,
  xlabel={Dimension $n$ (log scale)},
  ylabel={Bisection mean iterations},
  xmode=log,
  xmin=0.7, xmax=700,
  ymin=10, ymax=14,
  xtick={1,2,3,5,10,50,100,200,500},
  xticklabels={$1$,$2$,$3$,$5$,$10$,$50$,$100$,$200$,$500$},
  ytick={10,11,12,13,14},
  legend pos=north east,
  legend style={font=\small},
  grid=major,
  grid style={dashed,gray!30},
  thick
]
\addplot[color=blue, mark=square*, mark size=2.5pt, line width=1.2pt]
  coordinates {(1,12.00)(2,12.00)(3,12.00)(5,12.00)(10,12.00)
               (50,12.00)(100,12.00)(200,12.00)(500,12.00)};
\addlegendentry{Bisection (mean)}
\addplot[color=black, mark=none, line width=0.8pt, dotted,
  domain=0.7:700] {12};
\addlegendentry{$\lceil\log_2(2M)\rceil+1=12$}
\end{axis}
\end{tikzpicture}
\caption{Bisection mean iteration counts vs.\ dimension $n$
(log scale, $M=10^3$, $w=\tfrac{1}{2}$). The count is constant
at $12=\lceil\log_2(2000)\rceil+1$ across all
$n\in\{1,2,3,5,10,50,100,200,500\}$, confirming that bisection
complexity depends only on $M$ and is completely independent of
dimension, even at $n=500$.}
\label{fig:exp2}
\end{figure}
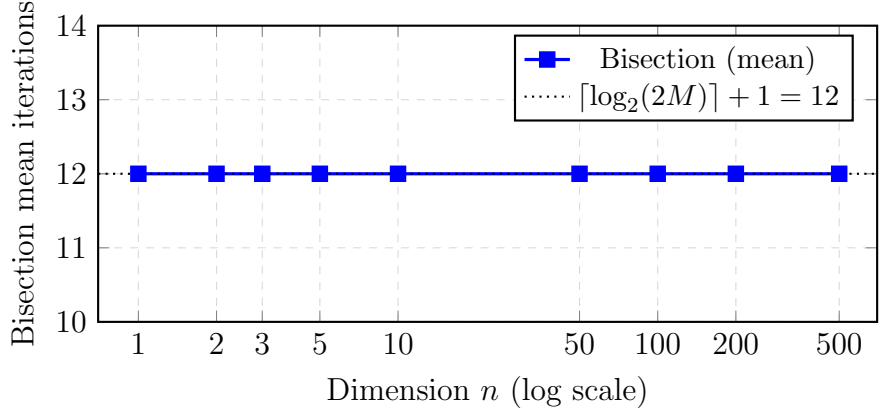

\section{Conclusion}
\label{sec:conclusion}

We have developed the foundations of tropical bi-objective pseudolinear 
optimization with two-sided constraints, connecting it to parametric 
mean-payoff games with two parameters. Our main results are:

\begin{enumerate}
  \item The feasibility region $\calR$ in parameter space is convex, and 
  the Pareto front $\calP$ is a convex piecewise-linear curve with finitely 
  many breakpoints (Theorem~\ref{thm:pareto_structure}).
  
  \item For integer data, all breakpoints of $\calP$ have coordinates 
  that are integer multiples of $\tfrac{1}{2}$, extending the denominator 
  bound of~\cite{PSW2023} to the bi-objective setting.

  \item Optimality and infeasibility certificates are given in terms of 
  the cycle structure of the parametric mean-payoff game 
  (Propositions~\ref{prop:optimality} and~\ref{prop:infeasibility}).
  
  \item A directional bisection algorithm traces the full Pareto front 
  in a finite number of steps by reducing to scalar bisection problems, 
  and a Newton scheme traces the full Pareto front in at most 
  $|\mathcal{S}|$ steps by jumping directly between breakpoints via 
  $2\times 2$ linear solves (Algorithms~\ref{alg:directional_bisection} 
  and~\ref{alg:newton}), here $|S|\leq 2^n(2+m)$ is independent of $M$ but exponential in $n$ (see Remark 9.5).
\end{enumerate}


\begin{thebibliography}{99}

\bibitem{AGG2012}
M.~Akian, S.~Gaubert, and A.~Guterman,
\textit{Tropical polyhedra are equivalent to mean payoff games},
\emph{Int. J. Algebra Comput.}, \textbf{22}(01) (2012), 1250001.
https://doi.org/10.1142/S0218196711006674

\bibitem{AB2012}
A.~Aminu and P.~Butkovi\v{c},
\textit{Non-linear programs with max-linear constraints: a heuristic approach},
\emph{IMA J. Manag. Math.}, \textbf{23}(1) (2012), 41--66.
https://doi.org/10.1093/imaman/dpq020

\bibitem{BA2009}
P.~Butkovi\v{c} and A.~Aminu,
\textit{Introduction to max-linear programming},
\emph{IMA J. Manag. Math.}, \textbf{20}(3) (2009), 233--249.
https://doi.org/10.1093/imaman/dpn032

\bibitem{BCOQ1992}
F.~Baccelli, G.~Cohen, G.J.~Olsder, and J.-P.~Quadrat,
\textit{Synchronization and Linearity: An Algebra for Discrete Event Systems},
John Wiley \& Sons, Chichester, 1992.

\bibitem{B2010}
P.~Butkovi\v{c},
\textit{Max-Linear Systems: Theory and Algorithms},
Springer, London, 2010.

\bibitem{CG1979}
R.A.~Cuninghame-Green,
\textit{Minimax Algebra},
Lecture Notes in Economics and Mathematical Systems, Vol.~166, 
Springer-Verlag, Berlin, 1979.

\bibitem{DS2004}
M.~Develin and B.~Sturmfels,
\textit{Tropical convexity},
\emph{Doc. Math.}, \textbf{9} (2004), 1--27.
https://doi.org/10.4171/dm/178

\bibitem{EM1979}
A.~Ehrenfeucht and J.~Mycielski,
\textit{Positional strategies for mean payoff games},
\emph{Int. J. Game Theory}, \textbf{8}(2) (1979), 109--113.
https://doi.org/10.1007/BF01768705

\bibitem{GKS2012}
S.~Gaubert, R.D.~Katz, and S.~Sergeev,
\textit{Tropical linear-fractional programming and parametric mean payoff games},
\emph{J. Symb. Comput.}, \textbf{47}(12) (2012), 1447--1478.
https://doi.org/10.1016/j.jsc.2011.12.049

\bibitem{GG1998}
S.~Gaubert and J.~Gunawardena,
\textit{The duality theorem for min-max functions},
\emph{C.R.A.S. Paris, S\'{e}rie I}, \textbf{326} (1998), 43--48.
https://doi.org/10.1016/S0764-4442(98)80017-X

\bibitem{JL2016}
M.~Joswig and G.~Loho,
\textit{Weighted digraphs and tropical cones},
\emph{Linear Algebra Appl.}, \textbf{501} (2016), 304--343.
https://doi.org/10.1016/j.laa.2016.03.023

\bibitem{K1980}
E.~Kohlberg,
\textit{Invariant half-lines of nonexpansive piecewise-linear transformations},
\emph{Math. Oper. Res.}, \textbf{5}(3) (1980), 366--372.
https://doi.org/10.1287/moor.5.3.366

\bibitem{Karp1978}
R.M.~Karp,
\textit{A characterization of the minimum cycle mean in a digraph},
\emph{Discrete Math.}, \textbf{23}(3) (1978), 309--311.
https://doi.org/10.1016/0012-365X(78)90011-0

\bibitem{K2012}
N.~Krivulin,
\textit{A new algebraic solution to multidimensional minimax location problems with Chebyshev distance},
\emph{WSEAS Trans. Math.}, \textbf{11}(7) (2012), 605--614.

\bibitem{K2014a}
N.~Krivulin,
Complete solution of a constrained tropical optimization problem with application to location analysis,
in: \emph{Relational and Algebraic Methods in Computer Science}, LNCS 8428,
Springer, 2014, pp.~362--378.
https://doi.org/10.1007/978-3-319-06251-8\_22

\bibitem{K2014b}
N.~Krivulin,
A constrained tropical optimization problem: complete solution and application example,
in: \emph{Tropical and Idempotent Mathematics and Applications}, Vol.~616, AMS,
Providence, RI, 2014, pp.~163--177.

\bibitem{K2015a}
N.~Krivulin,
\textit{Extremal properties of tropical eigenvalues and solutions to tropical optimization problems},
\emph{Linear Algebra Appl.}, \textbf{468} (2015), 211--232.
https://doi.org/10.1016/j.laa.2014.06.044

\bibitem{K2015b}
N.~Krivulin,
\textit{A multidimensional tropical optimization problem with a non-linear objective function and linear constraints},
\emph{Optimization}, \textbf{64}(5) (2015), 1107--1129.
https://doi.org/10.1080/02331934.2013.840624

\bibitem{K2017a}
N.~Krivulin,
\textit{Using tropical optimization to solve constrained minimax single facility location problems with rectilinear distance},
\emph{Comput. Manage. Sci.}, \textbf{14}(4) (2017), 493--518.
https://doi.org/10.1007/s10287-016-0260-6

\bibitem{K2017b}
N.~Krivulin,
\textit{Direct solution to constrained tropical optimization problems with application to project scheduling},
\emph{Comput. Manage. Sci.}, \textbf{14}(1) (2017), 91--113.
https://doi.org/10.1007/s10287-016-0259-0

\bibitem{PSW2023}
J.~Parsons, S.~Sergeev, and H.~Wang,
\textit{Tropical pseudolinear and pseudoquadratic optimization as parametric mean-payoff games},
\emph{Optimization}, \textbf{72}(11) (2023), 2793--2822.
https://doi.org/10.1080/02331934.2022.2124617

\bibitem{Z1976}
K.~Zimmermann,
\emph{Extremaln\'i algebra},
\'Utvar v\v{e}deck\'ych informac\'i ekonomick\'eho \'ustavu \v{C}SAV,
Prague, 1976.

\bibitem{Z1981}
U.~Zimmermann,
\emph{Linear and Combinatorial Optimization in Ordered Algebraic Structures},
Annals of Discrete Mathematics, Vol.~10, North Holland, Amsterdam, 1981.

\end{thebibliography}
\end{document}